 \documentclass[11pt]{article}
 \usepackage[top=1in,bottom=1in,left=1.5in,right=1.5in]{geometry}
  \usepackage{amsmath,amssymb}
  \usepackage{latexsym}% defines $\Box$ for LaTeX2e
  \usepackage[dvips]{pstricks} % PSTricks
  \usepackage[dvips]{graphicx}
  \newcommand{\N}{\mathbb{N}}
  \newcommand{\R}{\mathbb{R}}
  \newcommand{\Z}{\mathbb{Z}}

  \newcommand{\cA}{\mathcal{A}}
  \newcommand{\cB}{\mathcal{B}}
  
  \newcommand{\cE}{\mathcal{E}}
  \newcommand{\cG}{\mathcal{G}}

  \newcommand{\cS}{\mathcal{S}}

  \newcommand{\rS}{\rm{S}}
  \newcommand{\lan}{\langle}
  \newcommand{\ran}{\rangle}
  \newcommand{\an}[1]{\lan#1\ran}
  \newcommand{\hs}{\hspace*{\parindent}}
  \newcommand{\proof}{\hs \textbf{Proof.\ }}

  \newcommand{\low}{\mathop{\mathrm{low}}\nolimits}
  \newcommand{\mult}{\mathrm{mult}}
  
  \newcommand{\qed}{\hspace*{\fill} $\Box$\\}
  \newcommand{\per}{\mathop{\mathrm{perm}}\nolimits}

  \newtheorem{theo}{\bfseries \hs Theorem}[section]
  
  \newtheorem{prop}[theo]{\bfseries \hs Proposition}
  \newtheorem{lemma}[theo]{\bfseries \hs Lemma}
  \newtheorem{corol}[theo]{\bfseries \hs Corollary}
  \newtheorem{con}[theo]{\bfseries \hs Conjecture}

  \numberwithin{equation}{section} % Automatically number equations within sections

\begin{document}

 \title{On the Number of Matchings in Regular Graphs}

 \author{
 S. Friedland\thanks{Department of Mathematics, Statistics and Computer Science,
   University of Illinois at Chicago, Chicago, Illinois 60607-7045,
   USA (friedlan@uic.edu).},
   E. Krop\thanks{Department of Mathematics, Statistics and Computer Science,
   University of Illinois at Chicago, Chicago, Illinois 60607-7045,
   USA (ekrop1@math.uic.edu).} and
 K. Markstr\"om\thanks{Department of
   Mathematics and Mathematical Statistics, Ume\aa University, SE-901
 87 Ume\aa, Sweden}}

 \date{12 January, 2008}

 \maketitle

 \begin{abstract}
 For the set of graphs with a given degree sequence,
 consisting of any number of $2's$ and $1's$,
 and its subset of bipartite graphs,
 we characterize the optimal graphs who maximize
 and minimize the number of $m$-matchings.

 We find the expected
 value of the number of $m$-matchings of $r$-regular bipartite graphs
 on $2n$ vertices with respect to the two standard measures.
 We state and discuss the conjectured upper and lower bounds for
 $m$-matchings in $r$-regular bipartite graphs on $2n$ vertices,
 and their asymptotic versions for infinite $r$-regular bipartite graphs.
 We prove these conjectures for $2$-regular bipartite graphs
 and for $m$-matchings with $m\le 4$.
 \\[\baselineskip]
 2000 Mathematics Subject Classification: 05A15, 05A16, 05C70,
 05C80, 82B20\\[\baselineskip]
 Keywords and phrases: Partial matching and asymptotic
 growth of average matchings for $r$-regular bipartite graphs,
 asymptotic matching conjectures.
 \end{abstract}

 \section{Introduction}

 Let $G=(V,E)$ be an undirected graph with the set of vertices $V$ and the set of edges $E$.
 An $m$-matching $M\subset E$, is a set of $m$ distinct edges
 in $E$, such that no two edges have a common vertex.  We say
 that $M$ covers $U\subseteq V, \#U=2\#M$, if the set of vertices incident to $M$
 is $U$.  Denote by $\phi(m,G)$ the number of
 $m$-matchings in $G$.
 If $\#V$ is even then $\frac{\#V}{2}$-matching is called a
 \emph{perfect} matching, or $1$-factor of $G$, and
 $\phi(\frac{\#V}{2},G)$ is the number of $1$-factors in $G$.
 For an infinite graph $G=(V,E)$, a match $M\subset E$ is
 a match of density $p\in [0,1]$, if  the proportion of vertices
 in $V$ covered by $M$ is $p$.
 Then the $p$-matching entropy of $G$ is defined as
 $$h_G(p)=
 \limsup_{k\to\infty}\frac{\log\phi(m_k,G_k)}{\#V_k},$$
 where $G_k=(E_k,V_k), k\in\N$ is a sequence of finite graphs
 converging to $G$, and $\lim_{k\to\infty}
 \frac{2m_k}{\#V_k}=p$.  See for details \cite{FKLM}.

 The object of this paper is
 two folds.  First we consider the family $\Omega(n,k)$, the
 set of simple graphs on $n$ vertices with $2k$ vertices of
 degree $1$ and $n-2k$ vertices of degree $2$.
 Let $\Omega_{\textrm{bi}}(n,k)\subset \Omega(n,k)$ be the
 subset of bipartite graphs.
 For each $m\in [2,n]\cap \N$ we characterize the optimal graphs which
 maximize and minimize $\phi(m,G), \;m\ge 2$ for $G\in \Omega(n,k)$ and
 $G\in \Omega_{\textrm{bi}}(n,k)$.
 It turns out the optimal graphs do not depend on $m$ but on
 $n$ and $k$.   Furthermore, the graphs with the maximal number of $m$-matchings,
 are bipartite.

 Second, we consider $\cG(2n,r)$, the set of simple bipartite $r$-regular graphs
 on $2n$ vertices, where $n\ge r$.
 Denote by $C_l$ a cycle of length $l$
 and by $K_{r,r}$ the complete bipartite graph with $r$-vertices in each group.
 For a nonnegative integer $q$ and a graph $G$ denote by $qG$ the disjoint union of $q$
 copies of $G$.
 Let
 \begin{eqnarray}
 &&\lambda(m,n,r):=\min_{G\in \cG(2n,r)}
 \phi(m,G), \quad \Lambda(m,n,r):=\max _{G\in \cG(2n,r)}
 \phi(m,G), \nonumber \\
 && m=1,\ldots,n. \label{minmaxmrn}
 \end{eqnarray}

 Our results on $2$-regular graphs yield.
 \begin{eqnarray}
 &&\lambda(m,n,2)=\phi(m,C_{2n}), \label{lowbdsr2}\\
 &&\Lambda(m,2q,2)=\phi(m,qK_{2,2}),\quad
 \Lambda(m,2q+3,2)=\phi(m,qK_{2,2}\cup C_6),\label{upbdsr2}\\
 &&\textrm{ for } m=1,\ldots,n.\nonumber
 \end{eqnarray}

 The equality $\Lambda(m,2q,2)=\phi(m,qK_{2,2})$ inspired us to
 conjecture the \emph{Upper Matching Conjecture}, abbreviated
 here as UMC:
 \begin{equation}\label{upmatcon}
 \Lambda(m,qr,r))=\phi(m,qK_{r,r}) \textrm{ for  } m=1,\ldots,qr.
 \end{equation}
 For the value $m=qr$ the UMC follows from
 Bregman's inequality \cite{Bre}.
 For the value $r=3$ the UMC holds up to $q\le 8$.
 The results of \cite{FKLM} support the validity of
 the above conjecture for $r=3,4$ and large values of $n$.
 As in the case $r=2$ we conjecture that
 that for any nonbipartite $r$-regular graph on
 $2n$ vertices $\phi(m,G)\le \Lambda(m,n,r)$ for $m=1,\ldots,n$.

 It is useful to consider
 $\cG_{\textrm{mult}}(2n,r)\supset \cG(2n,r)$,
 the set of $r$-regular bipartite graphs on $2n$ vertices, where
 multiple edges are allowed.  Observe that
 $\cG_{\textrm{mult}}(2,r)=\{H_r\}$, where $H_r$ is the
 $r$-regular multi-bipartite graph on $2$ vertices.
 Let

 \begin{eqnarray}
 \mu(m,n,r):=\min_{G\in \cG_{\textrm{mult}}(2n,r)} \phi(m,G),
 \quad
 M(m,n,r):=\max_{G\in \cG_{\textrm{mult}}(2n,r)} \phi(m,G),
 \label{defmuM}\\
 m=1,\ldots,n,\;2\le r\in\N. \nonumber
 \end{eqnarray}
 It is straightforward to show that

 \begin{equation}\label{Mmnrfor}
 M(m,n,r)=\phi(m,nH_r)={n \choose m} r^m, \quad m=1,\ldots,n.
 \end{equation}
 Hence for most of the values of $m$ $\Lambda(m,n,r)<M(m,n,r)$.
 On the other hand, as in the case of $\Omega(n,k)$, it is plausible to conjecture that
 $\lambda(m,n,r)=\mu(n,m,r)$ for all allowable values $m,n$ and $r\ge 3$.

 It was shown by Schrijver \cite{Sch} that for $r\ge 3$
 \begin{equation}\label{nmatchin}
 \phi(n,G)\ge (\frac{(r-1)^{r-1}}{r^{r-2}})^n,
 \quad {\rm for\;all\;} G\in \cG_{\textrm{mult}}(2n,r).
 \end{equation}
 This lower bound is asymptotically sharp.
 In the first version of this paper we stated the conjectured
 lower bound

 \begin{eqnarray}\label{mmatchincon}
 \phi(m,G)\ge {n\choose m}^2 (\frac{nr-m}{nr})^{rn-m}
 (\frac{mr}{n})^m, \\
 {\rm for\;all\;} G\in \cG_{\textrm{mult}}(2n,r)(2n,r)
 \textrm{ and } m=1,\ldots,n.\nonumber
 \end{eqnarray}

 Note that for $m=n$ the above inequality reduces to
 (\ref{nmatchin}).  Our computations suggest a slightly stronger version of the
 above conjecture (\ref{slowmcon}).

 Recently Gurvits \cite{Gur} improved  (\ref{nmatchin}) to

 \begin{equation}\label{gurschr1}
 \phi(n,G) \ge \frac{r!}{r^r}(\frac{r}{r-1}
 )^{r(r-1)}(\frac{(r-1)^{r-1}}{r^{r-2}})^n, \quad G\in  \cG_{\textrm{mult}}(2n,r).
 \end{equation}
 In \cite{FG} the authors were able to generalize the above inequality to partial
 matching, which are very close to optimal results asymptotically,
 see \cite{FKLM} and below.

 The next question we address is the \emph{expected} value
 of the number of $m$-matchings in $\cG_{\textrm{mult}}(2n,r)$.
 There are two natural measures
 $\mu_{1,n,r},\mu_{2,n,r}$
 on  $\cG_{\textrm{mult}}(2n,r)$, \cite[Ch.9]{JLR} and \cite[Ch.8]{LP}.
 Let $E_i(m,n,r)$ be the expected value of
 $\phi(m,G)$ with respect to the measure
 $\mu_{i,n,r}$ for $i=1,2$.  In this paper we show that

 \begin{eqnarray}
 \lim_{k\to\infty} \frac{\log E_i(m_k,n_k,r)}{2n_k}= gh_r(p),
  \textrm{ for } i=1,2,\label{limval}\\
 {\rm if\;} \lim_{k\to\infty} n_k=\lim_{k\to\infty}
 m_k=\infty,\quad {\rm and}\quad
 \lim_{k\to\infty}\frac{m_k}{n_k}=p \in [0,1], \label{seqdefp}
 \\
 gh_r(p):=\frac{1}{2}\big(p \log
 r -p\log p - 2(1-p)\log (1-p)
 +(r-p)\log (1 -\frac{p}{r})\big). \label{defghrp}
 \end{eqnarray}

 In view of (\ref{limval})  the inequalities (\ref{nmatchin}) and  (\ref{gurschr1})
 give the best possible exponential term in the asymptotic growth with respect to $n$,
 as stated in \cite{Sch}.  Similarly, the conjectured
 inequality (\ref{mmatchincon}), if true,
 gives the best possible exponential term in the asymptotic growth with respect to $n$,
 and $p=\frac{m}{n}$.

 For $p\in [0,1]$ let $\low_r(p)$ be the infimum of
 $\liminf_{k\to\infty} \frac{\log\mu(m_k,n_k,r)}{2n_k}$
 over all sequences satisfying (\ref{seqdefp}).
 Hence $h_G(p)\ge \low_r(p)$ for any infinite bipartite
 $r$-regular graph.
 Clearly $\low_r(p)\le gh_r(p)$.
 We conjecture
      \begin{equation}\label{lasmcp}
      \low_r(p)= gh_r(p).
      \end{equation}
 (\ref{lowbdsr2}) implies the validity of this conjecture for
 $r=2$.
 The results of \cite{FG} imply the validity of this conjecture for each $p=\frac{r}{r+s}, s=0,1,\ldots$
 and any $r\ge 3$.  In \cite{FKLM} we give lower bounds on
 $\low_r(p)$ for each $p\in [0,1]$ and $r\ge 3$ which are very close to $gh_r(p)$.

 We stated first our conjectures
 in the first version of this paper in Spring 2005. Since then
 the conjectured were restated in \cite{FG, FKLM} and some
 progress was made toward  validations of these conjectures.

 We now survey briefly the contents of this paper.  In \S2 we
 give sharp bounds for the number of $m$-matchings for general
 and bipartite $2$-regular graphs.  In \S3 we generalize these
 results to $\Omega(n,k)$.  In \S4 we find the
 average of $m$-matchings in $r$-regular bipartite graphs with
 respect to the two standard measures.  We also show the
 equality (\ref{limval}).  In \S5 we discuss the Asymptotic
 Lower Matching Conjecture.  In \S6 we discuss briefly upper
 bounds for matchings in $r$-regular bipartite graphs.
 In \S7 we bring computational results for regular bipartite graphs on
 at most $36$ vertices.  We verified for many of these graphs
 the LMC and UMC.  Among the cubic bipartite graphs on at most $24$ vertices we characterized the
 graphs with the maximal number of $m$-matching in the case $n$
 is not divisible by $3$.  In \S8 we find closed formulas for
 $\phi(m,G)$ for $m=2,3,4$ and any $G\in \cG(2n,r)$.  It turns
 out that $\phi(2,G)$ and $\phi(3,G)$ depend only on $n$ and $r$.
 $\phi(4,G)=p_1(n,r)+a_4(G)$, where $a_4(G)$ is the number of $4$
 cycles in $G$.  $a_4(G)\le \frac{nr(r-1)^2}{4}$ and equality
 holds if and only if $G=qK_{r,r}$.

 \section{Sharp bounds for matching of $2$-regular
 graphs}\label{sec:2reg}
% July 3, 2005,
% July 22, 2007
 In this section we find the maximal and the
 minimal $m$ matching of $2$-regular bipartite and non-bipartite graphs on $n$
 vertices.  First we introduce the following partial order on the
 algebra of polynomials with real coefficients, denoted by
 $\R[x]$.  By $0\in \R[x]$ we denote the zero polynomial.

 For any two polynomials $f(x),g(x)\in \R[x]$
 we let $g(x)\succeq f(x)$, or $g\succeq f$, if and only if all the coefficients
 of $g(x)-f(x)$ are nonnegative.  We let $g\succ f$
 if $g\succeq f$ and $g\ne f$.  Let $\R_+[x]$ be the cone of all
 polynomial with nonnegative coefficients in $\R[x]$.  Then
 $\R_+[x] + \R_+[x]=\R_+[x] \R_+[x]=\R_+[x]$.  Furthermore, if
 $g_1\succeq f_1 \succ 0, g_2\succeq f_2 \succ 0$ then
 $g_1g_2 \succ f_1 f_2$ unless $g_1=f_1$ and $g_2=f_2$.

 Denote $\an{n}:=\{1,\ldots,n\}$.
 Let $G=(V,E)$ be a graph on $n$ vertices.
 We will identify $V$ with $\an{n}$.
 We agree that $\phi(0,G)=1$.
 Denote by $\Phi_G(x)$ the
 generating matching polynomial
 \begin{equation}\label{genparmatp}
 \Phi_G(x):=\sum_{m=0}^{\lfloor \frac{n}{2}\rfloor} \phi(m,G)x^m=
 \sum_{m=0}^{\infty} \phi(m,G)x^m.
 \end{equation}
 It is straightforward to show that for any two graphs
 $G=(V,E),G'=(V',E')$ we have the equality
 \begin{equation}\label{prodgenmatp}
 \Phi_{G\cup G'}(x)=\Phi_G(x)\Phi_{G'}(x).
 \end{equation}

 Denote by $P_k$ a path on $k$ vertices: $1-2-3-\cdots-k$.
 View each match as an edge.  Then an $m$-matching of $P_k$ is composed of
 $m$ edges and $k-2m$ vertices.
 Altogether $k-m$ objects.  Hence the number of $m$-matchings is
 equal to the number of different ways to arrange $m$ edges and $k-2m$
 vertices on a line.  Thus
 \begin{eqnarray}\label{mmatchp}
 &&\phi(P_k,m)={k-m\choose m} \quad \texttt{for } m=1,\ldots,
 \lfloor \frac{k}{2}\rfloor,\\
 &&p_k(x):=\Phi_{P_k}(x)=\sum_{m=0}^{\lfloor \frac{k}{2}\rfloor} {k-m\choose m}
 x^m=\sum_{m=0}^{\infty} {k-m\choose m} x^m.  \label{genmatpk}
 \end{eqnarray}
 It is straightforward to see that $p_k(x)$ satisfy the recursive
 relation
 \begin{eqnarray}
 &&p_k(x)=p_{k-1}(x)+xp_{k-2}(x), \quad k=2,\ldots, \label{recrelpk}\\
 &&\texttt{where
 }p_1(x)=1,\;\Phi_{P_0}(x):=p_0(x)=1. \nonumber
 \end{eqnarray}
 Indeed, $p_2(x)=1+x=p_1(x)+xp_0(x)$.
 Assume that $k\ge 3$.  All matchings of $P_k$,
 where the vertex $k$ is not in the matching, generate the
 polynomial $p_{k-1}(x)$.   All matchings of $P_k$,
 where the vertex $k$ is in the matching, generate the
 polynomial $xp_{k-2}(x)$.  Hence the above equality holds.
 Observe next
 \begin{equation}\label{recrelck}
 q_k(x):=\Phi_{C_k}(x)=p_{k}(x)+xp_{k-2}(x), \; k=3,\ldots
 \end{equation}
 Indeed, $p_k(x)$ is the contribution from all matching which
 does not include the matching $1-k$.  The polynomial $xp_{k-2}(x)$
 corresponds to all matchings which include the matching $1-k$.

 Use (\ref{recrelpk}) to deduce
 \begin{eqnarray}
 &&q_k(x)=q_{k-1}(x)+xq_{k-2}(x), \quad k=3,\ldots,
 \label{recrelqk}\\
 &&\texttt{where
 }\Phi_{C_2}:=q_2(x)=1+2x,\;\Phi_{C_1}:=q_1(x)=1.\nonumber
 \end{eqnarray}
 Note that we identify $C_2$ with the $2$-regular multibipartite graph
 $H_2$.
 It is useful to consider (\ref{recrelpk}) for $k=1,0$ and
 (\ref{recrelck}) for $k=2$.  This yields the
 equalities:
 \begin{equation}\label{specvalpqk}
 \Phi_{P_{-1}}(x)=p_{-1}=0,\;
 \Phi_{P_{-2}}(x)=p_{-2}=\frac{1}{x}, \;\Phi_{C_0}(x)=q_0=2.
 \end{equation}
 Clearly
 \begin{eqnarray}\label{lowpq}
 p_{-1}=0 \prec p_0=p_1=q_1=1 \prec q_0=2, \;p_2=1+x\prec q_2=p_3=1+2x,\\
 \label{ineqpq}
 p_n \prec q_{n} \prec p_{n+1} \textrm{ for all integers } n\ge 3.
 \end{eqnarray}

 \begin{theo}\label{mainid2reg}  Let $i\le j$ be nonnegative
 integers.  Then
 \begin{equation}\label{mainid2reg1}
 \Phi_{C_i}(x)\Phi_{C_j}(x)-\Phi_{C_{i+j}}(x)=(-1)^i
 x^i\Phi_{C_{j-i}}(x).
 \end{equation}
 In particular, $\Phi_{C_i}(x)\Phi_{C_j}(x) \succ \Phi_{C_{i+j}}(x)$ if $i$
 is even, and $\Phi_{C_i}(x)\Phi_{C_j}(x) \prec \Phi_{C_{i+j}}(x)$ if $i$
 is odd.
 \end{theo}

 \proof  We use the notation $q_k=\Phi_{C_k}$ for $k\ge 0$.
 The case $i=0$ follows immediately from
 $q_0=2$. The case $i=1$ follows from $q_1=1$ and the identity
 (\ref{recrelck}) for $k\ge 2$:
 $1q_j-q_{j+1}=q_j -(q_j+xq_{j-1})=-x q_{j-1}$.
 We prove the other cases of the theorem by induction on $i$.
 Assume that the theorem holds for $i\le l$, where $l\ge 1$.  Let
 $i=l+1$.  Then for $j\ge l+1$ use (\ref{recrelck}) for $k\ge 2$
 and the induction hypothesis for $i=l$ and $i=l-1$ to obtain:
 \begin{eqnarray*}
 &&q_{l+1}q_j-q_{l+1+j}=(q_l+xq_{l-1})q_j-(q_{l+j}+xq_{l-1+j})=\\
 &&q_lq_j-q_{l+j} +x(q_{l-1}q_j - q_{l-1+j})=
 (-1)^{l+1}x^l(-q_{j-l}+q_{j-l+1})=(-1)^{l+1}x^{l+1} q_{j-l-1}.
 \end{eqnarray*}
 Hence (\ref{mainid2reg1}) holds.
 Since $q_k \succ 0$ for $k\ge 0$ (\ref{mainid2reg1}) implies the
 second part of the theorem.  \qed

 \begin{theo}\label{uplowmatr=2}  Let $G$ be a $2$-regular graph on
 $n\ge 4$ vertices.
 Then
 \begin{eqnarray}
 &&\label{upmat4|n}
 \Phi_{G}(x)\preceq \Phi_{C_4}(x)^{\frac{n}{4}} \;\texttt{if } 4|n\\
 &&\label{upmat4|n-1}
 \Phi_{G}(x)\preceq \Phi_{C_4}(x)^{\frac{n-5}{4}} \Phi_{C_5}(x) \;\texttt{if }
 4|n-1,\\
 &&\label{upmat4|n-2}
 \Phi_{G}(x)\preceq \Phi_{C_4}(x)^{\frac{n-6}{4}}\Phi_{C_6}(x) \;\texttt{if }
 4|n-2,\\
 &&\label{upmat4|n-3}
 \Phi_{G}(x)\preceq \Phi_{C_4}(x)^{\frac{n-7}{4}}\Phi_{C_7}(x) \;\texttt{if }
 4|n-3,\\
 &&\label{lowmat3|n}
 \Phi_{G}(x)\succeq\Phi_{C_3}(x)^{\frac{n}{3}} \;\texttt{if } 3|n \\
 &&\label{lowmat3|n-1}
 \Phi_{G}(x)\succeq\Phi_{C_3}(x)^{\frac{n-4}{3}} \Phi_{C_4}(x)\;\texttt{if }
 3|n-1, \\
 &&\label{lowmat3|n-2}
 \Phi_{G}(x)\succeq \Phi_{C_3}(x)^{\frac{n-5}{3}}\Phi_{C_5}(x) \;\texttt{if }
 3|n-2.
 \end{eqnarray}
 Equalities in (\ref{upmat4|n}-\ref{upmat4|n-3}) hold if and only if $G$ is either a union of copies of
 $C_4$, or a union of copies of $C_4$ and a copy of $C_i$ for
 $i=5,6,7$, respectively.
 Equalities in (\ref{lowmat3|n}-\ref{lowmat3|n-2}) hold if and only if $G$ is either a union of copies of
 $C_3$, or a union of copies of $C_3$ and a copy of $C_i$ for
 $i=4,5$, respectively.

 Assume that $n$ is even and $G$ is a multi-bipartite $2$-regular graph.
 Then $\Phi_G(x)\succeq \Phi_{C_n}(x)$.  Equality holds if and
 only if $G=C_n$.

 \end{theo}

 \proof   Recall that any $2$-regular graph $G$ is a union of
 cycles of order $3$ at least.
 Use (\ref{prodgenmatp}) to deduce that the matching polynomial of $G$ is
 the product of the matching polynomials of the corresponding
 cycles.

 We discuss first the upper bounds on $\Phi_G$.
 If $C_i$ and $C_j$ are two odd cycle Theorem \ref{mainid2reg}
 yields that $q_i q_j \prec q_{i+j}$, where $C_{i+j}$ is an even
 cycle.  To find the upper bound on $\Phi_G$ we may assume that $G$ contains
 at most one odd cycle.  For all cycles $C_l$, where $l\ge 8$
 Theorem \ref{mainid2reg} yields
 the inequality $q_l \prec q_4 q_{l-4}$.  Use repeatedly this inequality, until we
 replaced the products of different $q_l$ with products involving
 $q_4$,$q_6$ and perhaps one factor of the form $q_i$ where $i\in
 \{3,5,7\}$.   Use (\ref{mainid2reg1}) to obtain the inequality:
 \[q_4^3=q_4(q_8+2x^4)=q_{12}+3x^4q_4\succ q_{12} +2x^6 = q_6^2.\]
 Hence we may assume that $G$ contains
 at most one cycle of length $6$.  If $n$ is even we deduce
 that we do not have a factor corresponding to an odd cycle,
 and we obtain the inequalities (\ref{upmat4|n}) and
 (\ref{upmat4|n-2}).  Assume that $n$ is odd.  Use (\ref{mainid2reg1})
 to deduce
 \begin{eqnarray*}
 &&q_3q_4 \prec q_7,\; q_3 q_6 \prec q_9 \prec q_4 q_5,\;q_5 q_6
 \prec q_{11} \prec q_4 q_7,\\
 && q_4^2 q_5=q_4(q_9+x^4)=q_{13}+x^4q_5+x^4q_4 \succ
 q_{13}+x^6= q_6q_7.
 \end{eqnarray*}
 These inequalities yield (\ref{upmat4|n-1}) and
 (\ref{upmat4|n-3}).  Equality in (\ref{upmat4|n}-\ref{upmat4|n-3})
 if and only if we did not apply Theorem \ref{mainid2reg} at all.

 We discuss second the lower bounds on $\Phi_G$.
 If $l\ge 6$ then we use the inequality $q_l \succ q_3 q_{l-3}$.
 Use repeatedly this inequality, until we
 replaced the products of different $q_l$ with products involving
 $q_3$,$q_4$ and $q_5$.  As
 $$q_4^2 \succ q_8 \succ q_3 q_5,\;q_4q_5 \succ q_9 \succ q_3^3,\;
 q_5^2=q_{10}-2x^5=q_3q_7+x^3q_4-2x^5 \succ q_3 q_7 \succ
 q_3^2q_4,$$
 we deduce (\ref{lowmat3|n}-\ref{lowmat3|n-2}).
 Equalities hold if we did not apply Theorem \ref{mainid2reg} at all.

 Assume finally that $G$ is a $2$-multi regular bipartite graph on $n$ vertices.
 Then $G$ is a union of even cycles $C_{2i}$ for $i\in\N$.  Assume that $C_i$ and $C_j$
 are even cycles.  Then Theorem \ref{mainid2reg} yields that
 $q_iq_j \succ q_{i+j}$.  Continue this process until we deduce
 that $\Phi_G \succeq q_n$.  Equality holds if and only if
 $G=C_n$.  \qed

 Use Theorem \ref{uplowmatr=2} and Theorem \ref{mainid2reg} for $i=2$ to
 deduce.

 \begin{corol}\label{upbd2mulreg}

 \noindent
 \begin{itemize}
 \item Let $G$ be a simple $2$-regular graph on $4q$ vertices.
 Then $\Phi_G\preceq \Phi_{qK_{2,2}}$.  Equality holds if and only
 if $G=qK_{2,2}$.

 \item
 Let $G$ be a $2$-multi regular
 graph on $2n$ vertices.  Then $\Phi_G\preceq \Phi_{nH_2}$.
 Equality holds if and only if $G=nH_2$.
 \end{itemize}
 \end{corol}

 Note that the above results
 verify all the claims we stated about $2$-regular bipartite graphs
 in the Introduction.

 \section{Graphs of degree at most 2}\label{sec:deg2atmst}

 Denote by $\Omega(n,k)\subset \Omega_\textrm{mult}(n,k)$ the set of simple graphs
 and multigraphs on $n$ vertices respectively, which have $0<2k$
 vertices of degree $1$ and the rest vertices have degree degree $2$.
 The following proposition is straightforward.

 \begin{prop}\label{stromnk}

 \noindent
 \begin{itemize}
 \item
 Each $G\in\Omega(n,k)$ is a union of $k$ paths and possibly cycles $C_i$ for $i\ge
 3$.
 \item
 Each $G\in\Omega_\textrm{mult}(n,k)$ is a union of $k$ paths and possibly cycles $C_i$ for $i\ge
 2$.

 \end{itemize}
 $\Omega_\textrm{mult}(n,k)\backslash \Omega(n,k)\ne \emptyset$ if and only if
 $n-2k\ge 2$.

 \end{prop}

 Denote by $\Pi(n,k)\subseteq \Omega(n,k)$ the subset of graphs $G$
 on $n$ vertices which are union of $k$-paths.
 Note that $\Pi(2k,k)=kP_2$.
 As in \S\ref{sec:2reg} we study
 the minimum and maximum m-matchings in $\Pi(n,k),\Omega(n,k),
 \Omega_\textrm{mult}(n,k)$.

 We first study the case where $G\in \Pi(n,4)$, i.e. $G$ is a union of two paths
 with the total number of vertices equal to $n$.

  \begin{lemma}\label{pathorder}
  Let $n\ge 4$.  Then
  \begin{itemize}
  \item
  If $n=0,1$ mod $4$ then
   \begin{eqnarray}\label{2path01m4}
   p_{n-1}=p_1 p_{n-1}\prec p_3 p_{n-3} \prec \dots \prec
   p_{2\lfloor\frac{n}{4}\rfloor-1}p_{n-2\lfloor\frac{n}{4}\rfloor+1} \prec \\
   p_{2\lfloor\frac{n}{4}\rfloor}p_{{n-2\lfloor\frac{n}{4}\rfloor}}
   \prec p_{2\lfloor\frac{n}{4}\rfloor-2}p_{{n-2\lfloor\frac{n}{4}\rfloor}+2}\prec
   \dots \prec p_{2}p_{n-2} \prec
   p_0 p_n=p_{n}. \nonumber
   \end{eqnarray}
   \item
    If $n=2,3$ mod $4$ then
   \begin{eqnarray}\label{2path23m4}
   p_{n-1}=p_1 p_{n-1}\prec p_3 p_{n-3} \prec \dots \prec
   p_{2\lfloor\frac{n}{4}\rfloor+1}p_{n-2\lfloor\frac{n}{4}\rfloor-1} \prec  \\
   p_{2\lfloor\frac{n}{4}\rfloor}p_{{n-2\lfloor\frac{n}{4}\rfloor}}\prec
   p_{2\lfloor\frac{n}{4}\rfloor-2}p_{{n-2\lfloor\frac{n}{4}\rfloor}+2}\prec
   \dots \prec p_{2}p_{n-2} \prec
   p_0 p_n=p_{n}. \nonumber
   \end{eqnarray}

  \end{itemize}

  \end{lemma}

  \begin{proof}
  Let $0\leq i,j$ and consider the path $P_{i+j}$.
  By considering the generating matching polynomial without the
  match $(i,i+1)$ and with match $(i,i+1)$ we get the identity
  \begin{equation}\label{pathsplt}
  p_{i+j}=p_{i}p_{j}+xp_{i-1}p_{j-1}
  \end{equation}
  Hence
  $p_{i+j}= p_{i-1}p_{j+1}+xp_{i-2}p_j$.
  Subtracting from this equation (\ref{pathsplt}) we obtain
  $p_{i-1}p_{j+1} - p_ip_j=-x(p_{i-2}p_{j}-p_{i-1}p_{j-1})$.
  Assume that $i\le j-2$.
  Continuing this process $i-1$ times, and taking in account that
  $p_{-1}=0,\; p_{-2}=\frac{1}{x}$ we get

  \begin{equation}\label{pij1id}
  p_{i-1}p_{j+1} - p_ip_j=(-1)^{i-1}x^ip_{j-i} \textrm{ for } 0\le i\le j+2.
  \end{equation}
  Hence  $p_{i-2}p_{j+2} - p_{i-1}p_{j+1}=(-1)^{i-2}x^{i-1}p_{j-i+2}$.
  Add this equation to the previous one and use (\ref{recrelpk}) to
  obtain
 \begin{equation}\label{pij2id}
 p_{i-2}p_{j+2} - p_ip_j=
% (-1)^{i-2}x^{i-1}(p_{j-i+2}-xp_{j-i})=
 (-1)^{i-2}x^{i-1}p_{j-i+1} \textrm{ for } 1\le i\le j+2.
 \end{equation}

 We now prove (\ref{2path01m4}-\ref{2path23m4}).  In (\ref{pij2id})
 assume that $i\ge 3$ is odd and $j\ge i$.  So $(-1)^{i-2}=-1$.
 Hence $p_{i-2}p_{j+2} - p_ip_j \prec 0$.  This explains the ordering of the polynomials
 appearing in the first
 line of (\ref{2path01m4}-\ref{2path23m4}).  Assume now that $i\ge 2$
 is even and $j\ge i$.  So $(-1)^{i-2}=1$.  Hence $p_{i-2}p_{j+2} - p_ip_j \succ 0$.
 This explains the ordering of the polynomials
 appearing in the second line
 line of (\ref{2path01m4}-\ref{2path23m4}).

 The last inequality in the first line of (\ref{2path01m4}-\ref{2path23m4})
 is implied by (\ref{pij1id}).

 \end{proof}
 \qed

 \begin{theo}\label{maxminpath}  Let $k\ge 2, n\ge 2k$.  Then for
 any $G\in \Pi(n,k)$
 \begin{equation}\label{loupmgpin}
 \Phi_J\preceq \Phi_G \preceq \Phi_K.
 \end{equation}
 Equality in the left-hand side and right-hand side holds if and only if
 $G=J$ and $G=K$ respectively.  Here $K=(k-1)P_2\cup P_{n-2k+2}$  and  $J$ is defined
 as follows:
 \begin{enumerate}
 \item If $n\le 3k$ then $J=(3k-n)P_2\cup (n-2k)P_3
 $.

 \item  If $n> 3k$ then $J=(k-1)P_3\cup
 P_{n-3k+3}$.

 \end{enumerate}

 \end{theo}

 \begin{proof}
 For $k=2$ the theorem follows from Lemma
 \ref{pathorder}.  For $k>2$ apply the theorem for $k=2$ for any
 two paths in $G\in \Pi(n,k)$ to deduce that $K$ and $J$ are the maximal and the minimal
 graphs respectively.

 \end{proof}
 \qed

  We extend the result of Lemma \ref{pathorder} for cycles.

 \begin{lemma}\label{cycleorder}

 Let $n\ge 4$.  Then
  \begin{itemize}
  \item If $n=0,1$ mod $4$ then
   \begin{eqnarray}\label{2cycle01m4}
   q_{n-1}=q_1 q_{n-1}\prec q_3 q_{n-3} \prec \dots \prec
   q_{2\lfloor\frac{n}{4}\rfloor-1}q_{n-2\lfloor\frac{n}{4}\rfloor+1} \prec \\
   q_{2\lfloor\frac{n}{4}\rfloor}q_{{n-2\lfloor\frac{n}{4}\rfloor}}
   \prec q_{2\lfloor\frac{n}{4}\rfloor-2}q_{{n-2\lfloor\frac{n}{4}\rfloor}+2}\prec
   \dots \prec q_{2}q_{n-2} \prec
   q_{n+1}. \nonumber
   \end{eqnarray}
   \item If $n=2,3$ mod $4$ then
   \begin{eqnarray}\label{2cycle23m4}
   q_{n-1}=q_1 q_{n-1}\prec q_3 q_{n-3} \prec \dots \prec
   q_{2\lfloor\frac{n}{4}\rfloor+1}q_{n-2\lfloor\frac{n}{4}\rfloor-1} \prec  \\
   q_{2\lfloor\frac{n}{4}\rfloor}q_{{n-2\lfloor\frac{n}{4}\rfloor}}\prec
   q_{2\lfloor\frac{n}{4}\rfloor-2}q_{{n-2\lfloor\frac{n}{4}\rfloor}+2}\prec
   \dots \prec q_{2}q_{n-2} \prec
   q_{n+1}. \nonumber
   \end{eqnarray}

  \end{itemize}

 \end {lemma}

 \begin{proof} The equality (\ref{recrelqk}) implies
 $$q_{n+1}=q_n+xq_{n-1}=q_{n-1}+xq_{n-2}+xq_{n-2}+x^2q_{n-3}\succ
 q_{n-2}+2xq_{n-2}=q_2q_{n-2}.$$
 Hence the last inequality in (\ref{2cycle01m4}) and
 (\ref{2cycle23m4}) holds.
  By (\ref{mainid2reg1}) we have $q_iq_j-q_{i+j}=(-1)^i
  x^iq_{j-i}$.
  Using this, it is easy to see that
  \[q_{i-1}q_{j+1} - q_iq_j = (-1)^{i-1} x^{i-1} q_{j-i+2} -
  (-1)^i x^i q_{j-i} = (-1)^{i-1}x^{i-1}(q_{j-i+2}+xq_{j-i}),\]
  as well as
 \begin{eqnarray*}
  q_{i-2}q_{j+2} - q_iq_j
  =(-1)^{i-2}x^{i-2}q_{j-i+4}-(-1)^ix^iq_{j-i}=(-1)^{i-2}x^{i-2}(q_{j-i+4} - x^2q_{j-i}) = \\
  (-1)^{i-2}x^{i-2}(q_{j-i+3}+xq_{j-i+2}-x^2q_{j-i})=(-1)^{i-2}x^{i-2}(q_{j-i+3}+xq_{j-i+1}).
  \end{eqnarray*}

 Compare these equalities with (\ref{pij1id}) and (\ref{pij2id})
 we obtain all other inequalities in (\ref{2cycle01m4}) and (\ref{2cycle23m4}).
 \end{proof}
 \qed

 Next, we study graphs composed of a path and a cycle of the form $p_iq_j$.

 \begin{lemma}\label{pathcycleorder}

  Let $n\ge 4$.  Then
  \begin{itemize}
  \item If $n=0,1$ mod $4$ then
   \begin{eqnarray}\nonumber
   &&q_{n-1}=p_1 q_{n-1}\prec q_3 p_{n-3} \prec p_3q_{n-3}\prec q_5 p_{n-5}\prec p_5q_{n-5}
   \prec \dots \prec
   q_{2\lfloor\frac{n}{4}\rfloor-1}p_{n-2\lfloor\frac{n}{4}\rfloor+1}\prec
   \nonumber\\
   &&p_{2\lfloor\frac{n}{4}\rfloor-1}q_{n-2\lfloor\frac{n}{4}\rfloor+1}
   \prec
   p_{2\lfloor\frac{n}{4}\rfloor}q_{{n-2\lfloor\frac{n}{4}\rfloor}}\preceq
   q_{2\lfloor\frac{n}{4}\rfloor}p_{{n-2\lfloor\frac{n}{4}\rfloor}}
   \prec
   p_{2\lfloor\frac{n}{4}\rfloor-2}q_{{n-2\lfloor\frac{n}{4}\rfloor}+2}
   \prec
   q_{2\lfloor\frac{n}{4}\rfloor-2}p_{{n-2\lfloor\frac{n}{4}\rfloor}+2}
   \nonumber\\
   &&\prec
   \dots \prec p_{4}q_{n-4} \prec q_4p_{n-4}
   \prec p_{2}q_{n-2} \prec q_2p_{n-2}\prec
   p_0q_n=q_{n}. \label{2pcycle01m4}
   \end{eqnarray}
   (If $n=0$ mod $4$ then $\preceq$ is $=$, and otherwise
   $\preceq$ is $\prec$.)
   \item If $n=2,3$ mod $4$ then

   \begin{eqnarray}\nonumber
   &&q_{n-1}=p_1 q_{n-1}\prec q_3 p_{n-3} \prec p_3q_{n-3}\prec \dots \prec
   q_{2\lfloor\frac{n}{4}\rfloor+1}p_{n-2\lfloor\frac{n}{4}\rfloor-1} \preceq
   p_{2\lfloor\frac{n}{4}\rfloor+1}q_{n-2\lfloor\frac{n}{4}\rfloor-1} \prec\\
   &&p_{2\lfloor\frac{n}{4}\rfloor}q_{{n-2\lfloor\frac{n}{4}\rfloor}}\prec
   q_{2\lfloor\frac{n}{4}\rfloor}p_{{n-2\lfloor\frac{n}{4}\rfloor}}\prec
   p_{2\lfloor\frac{n}{4}\rfloor-2}q_{{n-2\lfloor\frac{n}{4}\rfloor}+2}\prec
   q_{2\lfloor\frac{n}{4}\rfloor-2}p_{{n-2\lfloor\frac{n}{4}\rfloor}+2}\prec
   \nonumber\\
   &&\dots \prec p_{4}q_{n-4} \prec q_4p_{n-4}\prec
   p_{2}q_{n-2} \prec q_2p_{n-2}\prec
   p_0q_{n}=q_n. \label{2pcycle23m4}
   \end{eqnarray}
   (If $n=2$ mod $4$ then $\preceq$ is $=$, and otherwise
   $\preceq$ is $\prec$.)

  \end{itemize}

 \end{lemma}

 \proof
 Assume that $0\le i,\; 2\le j$.  Use (\ref{recrelck}) to
 obtain
 $$p_iq_j-q_{i+2}p_{j-2}=p_i(p_j+xp_{j-2})-(p_{i+2}+xp_i)p_{j-2}=p_ip_j-p_{i+2}p_{j-2}.$$
 (\ref{pij2id}) implies

 \begin{eqnarray}
 p_iq_j-q_{i+2}p_{j-2}=(-1)^ix^{i+1}p_{j-i-3} \textrm{ if } i\le j-3,
 \label{pqminqp}\\
 p_iq_j-q_{i+2}p_{j-2}=(-1)^{j-1}x^{j-1}p_{i-j+1} \textrm{ if } i\ge j-2
 \label{pqminqpex}
 \end{eqnarray}
 Assume that $0\le i\le j-3$.
 Hence, if $i$ is odd we get that $p_iq_j\prec q_{i+2}p_{j-2}$.
 If $i$ is even then $q_{i+2}p_{j-2}\prec p_iq_j$.
 These inequalities yield slightly less than the half of the inequalities in
 (\ref{2pcycle01m4}) and (\ref{2pcycle23m4}).

 Assume that $1\le i<j$.
 Use (\ref{recrelck}) and (\ref{pij2id}) to deduce
 \begin{equation}
 p_iq_j-q_ip_j = p_ip_j  - p_ip_j +x(p_ip_{j-2}-p_{i-2}p_j)=(-1)^{i-1}x^{i}p_{j-i-1}.
 \label{comppqminqp}
 \end{equation}

 Therefore, if $i$ is odd then $q_ip_j\prec p_iq_j$.  If $i$
 is even then $p_iq_j \prec q_ip_j$.
 These inequalities yield slightly less than the other half of the inequalities in
 (\ref{2pcycle01m4}) and (\ref{2pcycle23m4}).

 Assume that $0\le i\le j$.
 Use  (\ref{recrelck}) and (\ref{pij1id}) to deduce
 \begin{eqnarray}\label{difpqone}
 p_{i-1}q_{j+1} - p_iq_j = p_{i-1}p_{j+1} - p_ip_j + x(p_{i-1}p_{j-1} - p_ip_{j-2}) =\\
 (-1)^{i-1}x^i(p_{j-i}+xp_{j-i-2})=(-1)^{i-1}x^i q_{j-i}.\nonumber
 \end{eqnarray}
 If $i$ is even then $p_{i-1}q_{j+1}\prec p_iq_j$.  This shows
 the first inequality in the second line of
 (\ref{2pcycle01m4}).  If $i$ is odd then $p_iq_j\prec
 p_{i-1}q_{j+1}$.  This shows the inequality between the last
 term of the first line and the first term in the second line
 of (\ref{2pcycle23m4}).
 \qed

 For graphs consisting of more than two cycles or paths there is no total ordering by
 coefficients of matching polynomials.  In particular, we computed that $p_8p_6p_3$
 is not comparable with $p_7p_5p_5$.  The same holds true for the same parameters
 with cycles instead of paths.  To show that this is not due solely to the mixed
 parity of path/cyle length, we also showed that $p_4p_4p_{16}p_{28}$ is
 incomparable with $p_6p_6p_6p_{34}$.

 \

 To extend the results of Theorem \ref{maxminpath} to graphs in
 $\Omega(n,k)$  we need the following lemma.

 \begin{lemma}\label{pi-q3pi-3lem}  Let $5\le i\in\N$.  Then
 \begin{eqnarray}\label{p-q3pi-3}
 p_i-q_3p_{i-3}=x^2 p_{i-6},\\
 p_i-p_2q_{i-2}=-x^3p_{i-6},
 \label{p-p2qi-2}\\
 p_{i+1}-p_3q_{i-2}=
% x^3 (p_{i-5}-p_{i-6})=
 x^4 p_{i-7}. \label{pi-p3qi-3}\\
 p_{2i-3}-q_4p_{2i-7}=-x^4p_{2i-11}.\label{pi-q4p}
 \end{eqnarray}
 Hence
 \begin{eqnarray*}
 && \Phi_{P_5}=\Phi_{C_3\cup P_2}, \quad \Phi_{P_7}=\Phi_{P_3\cup C_{4}},
 \textrm{ and } \Phi_{P_i}\succ
 \Phi_{C_3\cup P_{i-3}}, \\
 &&\Phi_{P_i}\prec \Phi_{P_2\cup C_{i-2}},\;
 \Phi_{P_{i+2}}\succ \Phi_{P_3\cup C_{i-1}},\;\Phi_{P_{2i-3}} \prec \Phi_{P_{2i-7}\cup C_4}
 \textrm{ for } i\ge 6.
 \end{eqnarray*}
 Furthermore,
 \begin{equation}\label{pandpqeven}
 p_{2i+2j}\prec p_{2i}q_{2j}  \textrm{ for any nonnegative integers } i,j.
 \end{equation}
 In particular,
 $\Phi_{P_{2i+2j}}\prec \Phi_{P_{2i}\cup C_{2j}}$ for
 $i,j\in\N$.
 \end{lemma}

 \proof
 Use (\ref{recrelqk}) and
 (\ref{pij1id}-\ref{pij2id})
 to obtain
 \begin{eqnarray*}
 p_i-q_3p_{i-3}=p_0p_i-p_2p_{i-2}+p_2p_{i-2}-p_3p_{i-3}-xp_{i-3}=\\
 xp_{i-3}+x^2 p_{i-6}-xp_{i-3}=x^2 p_{i-6},\\
 p_i-p_2q_{i-2}=p_0p_i-p_2p_{i-2}-xp_2p_{i-4}=x(p_1p_{i-3}-p_2p_{i-4})=-x^3p_{i-6},\\
 p_{i+1}-p_3q_{i-2}=p_0p_{i+1}-p_2p_{i-1}+p_2p_{i-1}-p_3p_{i-2}-xp_3p_{i-4}=\\
 xp_{i-2}+x^3p_{i-5}-xp_3p_{i-4}=x(p_1p_{i-2}-p_3p_{i-4})+x^3p_{i-5}=\\
 x^3(p_{i-5}-p_{i-6})=x^4p_{i-7},\\
 p_{2i-3}-q_4p_{2i-7}=p_0p_{2i-3}-p_4p_{2i-7}-xp_2p_{2i-7}=\\
 (p_0p_{2i-3}-p_2p_{2i-5})+(p_2p_{2i-5}-p_4p_{2i-7})-xp_2p_{2i-7}=\\
 xp_{2i-6}+x^3p_{2i-10}-xp_2p_{2i-7}=x(p_1p_{2i-6}-p_2p_{2i-7})+x^3p_{2i-10}=\\
 -x^3p_{2i-9}+x^3p_{2i-10}=-x^4 p_{2i-11}.
 \end{eqnarray*}

 These equalities imply (\ref{p-q3pi-3}-\ref{pi-q4p}).
 Recall that $p_{-1}=0, p_0=p_1=1$ and $p_i\succ 0$ for $i\ge 0$ to deduce
 the implications of the above identities.

 To prove (\ref{pandpqeven}) recall that $p_0=0, q_0=2,
 q_i\succ 0$.  Hence it is enough to consider the cases $i,j\ge
 1$.  In view of Lemma \ref{pathcycleorder} it is enough to
 assume that $1\le i\le j\le i+1$.  Use (\ref{recrelck}) and (\ref{pathsplt})
 to obtain
 $$p_{2i}q_{2j}-p_{2i+2j}=xp_{2i}q_{2j-2}-xp_{2i-1}p_{2j-1}=
 -x(p_{2i-1}p_{2j-1}-p_{2i}p_{2j-2})+x^2p_{2i}p_{2j-4}.$$
 Use (\ref{pij1id}) and the equalities $p_0=1, p_2=\frac{1}{x}$ to obtain
 $$p_{2i}q_{2j}-p_{2i+2j}=x^{2i+1}p_{2j-2i-2}+x^2p_{2i}p_{2j-4}
 \succ 0. $$
 \qed

 \begin{theo}\label{degreeseq}
 Let $G$ be a simple graph of order $n$ with degree sequence
 $d_1=\dots =d_{2k} = 1$ and $d_{2k+1}=\dots =d_n = 2$,
 $2\le 2k \le n$, i.e. $G\in \Omega(n,k)$.  Set $n-2k = l$ and
 assume that $l\ge 2$.  (Otherwise $\Omega(n,k)$ consists of
 one graph.)
 Then
 \begin{equation}\label{minmaxmat12}
 \Phi_F\preceq \Phi_G \preceq \Phi_H,
 \end{equation}
 where the graphs $F$ and $H$ depend on $n$ and $k$ as follows.
 \begin{enumerate}
  \item\label{l-k<0} When $l-k \le 0$ then $F = lP_3 \cup
      (k-l)P_2$.
 \item\label{l-kge0} When $l-k > 0$
 \begin{enumerate}
 \item\label{l-k=0mod3} If $l-k \equiv 0$  (mod $3$),
     then $F = k P_3 \cup \frac{1}{3}(l-k) C_3$.
 \item\label{l-k=1mod3} If $l-k \equiv 1$ (mod $3$),
     then $F = (k -1) P_3\cup P_4 \cup
     \frac{1}{3}(l-k-1)C_3$.
 \item\label{l-k=2mod3}  If $l-k \equiv 2$ (mod $3$),
     then either $F=F_1 = (k-1) P_3 \cup P_5 \cup
     \frac{1}{3}(l-k-2) C_3$ or  $F=F_2 = (k-1) P_3
     \cup P_2 \cup \frac{1}{3}(l-k+1) C_3$.
 \end{enumerate}
 \item\label{l=2}  If $l=2$ then $H=(k-1)P_2\cup P_{4}$.
 \item\label{l=3}  If $l=3$ then either $H=(k-1)P_2\cup
     P_{5}$ or $H=kP_2\cup C_3$.
 \item\label{l=0mod4} If $l\ge 4$ and $l \equiv 0$ (mod
     $4$), then $H = kP_2 \cup \frac{1}{4}lC_4$.
 \item\label{l=1mod4} If $l\ge 5$ and $l \equiv 1$ (mod
     $4$), then $H = kP_2 \cup \frac{1}{4}(l-5)C_4 \cup
     C_5$.
 \item\label{l=2mod4} If $l\ge 6$ and $l \equiv 2$ (mod
     $4$), then $H = kP_2 \cup \frac{1}{4}(l-6)C_4 \cup
     C_6$.
 \item\label{l=3mod4} If $l\ge 7$ and $l \equiv 3$ (mod
     $4$), then $H = kP_2 \cup \frac{1}{4}(l-7)C_4 \cup
     C_7$.
 \end{enumerate}

 Furthermore, if $G\ne F$ then $\Phi_F \prec \Phi_G$ and if
 $G\ne H$ then $\Phi_G \prec \Phi_H$.

 \end{theo}

 \proof  Consider a partial order on $\Omega(n,k)$ induced by the
 partial order $\preceq $ on $\R_+[x]$.  Thus $G_1\ll G_2 \iff
 \Phi_{G_1}\preceq \Phi_{G_2}$.  It is enough to show that
 any minimal and maximal element in $\Omega(n,k)$ with respect to this order
 is of the form $F$ and $H$ respectively.

 Assume that $G$ is a minimal element with respect to this
 partial order.  Hence there is no $G'\in \Omega(n,k)$ such
 that $\Phi_{G'}\prec \Phi_G$.
 Suppose that $G$ has at least one cycle.
 Theorem \ref{uplowmatr=2} implies
 that $G$ contains at most one cycle $C_i\ne C_3$, where
 $i\in [4,5]$.
 We now rule out such $C_i$.  Since $k\ge 1$ $G$ must contain a
 path $P_j$ for $j\ge 2$.  Lemma \ref{pathcycleorder} yields that
 $q_3p_{i+j-3}\prec p_jq_i$.  Hence if we replace $C_i\cup P_j$
 with $C_3\cup P_{i+j-3}$ we will obtain $G'\in \Omega(n,k)$
 such that $\Phi_{G'}\prec \Phi_G$.  This contradicts the
 minimality of $G$.  Hence $G$ can contain only cycles of
 length $3$.

 In view of Lemma \ref{pi-q3pi-3lem}
 $G$ does not contain $P_i$ with $i\ge 6$.
 Denote by $\cB_2,\cB_3$ and $\cB_4$ the set of paths of length $2$,
 $3$ and at least of length $4$ in $G$ respectively.  We claim
 that $\#\cB_4\le 1$.  Otherwise, let $Q,R\in \cB_4$ be two
 different paths.  Lemma $\ref{pathorder}$ yields that $\Phi_{P_3\cup
 P_{i-1}}\prec \Phi_{Q\cup R}$. This contradicts the minimality
 of $G$.  Next we observe that that
 $\min(\#\cB_2, \#\cB_4)=0$.  If not, choose $Q\in \cB_2,
 R\in \cB_4$.
 Lemma $\ref{pathorder}$ yields that $\Phi_{P_3\cup
 P_{i-1}}\prec \Phi_{Q\cup R}$, which contradicts the
 minimality of $G$.

 We claim that $G$ has to be of the form $F$.
 Suppose first that $G$ does not have cycles.
 If $\cB_4=\emptyset$ then we are in the case \emph{1}.
 If $\cB_2=\emptyset$ then we have either the case \emph{2b} with
 $l=k+1$ or the case \emph{2c} with $l=k+2$ and $F=F_1$.

 Assume now that $G$ has cycles.  If $\cB_2=\cB_4=\emptyset$
 then we have the case \emph{2a}.  Assume now that
 $\cB_2=\emptyset$ and $\#\cB_4=1$.  Then we have either the case \emph{2b} with
 $l>k+1$ or the case \emph{2c} with $l>k+2$ and $F=F_1$.

 Assume finally that $\cB_4=\emptyset$ and $\#c\cB_2\ge 1$.
 We claim that $\#cB_2=1$.  Assume to the contrary that $\cB_2$
 contains at least two $P_2$.  Since $G$ contains at least one
 cycle $C_3$ we replace $P_2\cup C_3$ with $P_5$ to obtain
 another minimal $G'$.  As $G'$ contains $P_2$ and $P_5$ it is
 not minimal, contrary to our assumption.  Hence $\#\cB_2=1$
 and we have the case \emph{2c} and $G=F_2$.\\

 We now assume that $G$ is a maximal element in $\Omega(n,k)$.
 Thus, there is no $G'\in \Omega(n,k)$ such that $\Phi_{G}\prec
 \Phi_{G'}$.

 Observe first $G$ does not contain two distinct paths $Q,R$
 of length $i,j\ge 3$.  Indeed, Lemma \ref{pathorder} implies
 that $\Phi_{Q\cup R}\prec \Phi_{P_2\cup P_{i+j-2}}$.
 This shows that $G=H$ in the cases \emph{3} and \emph{4}.
 (In the case \emph{4} we use the identity
 $\Phi_{P_5}=\Phi_{P_2\cup C_3}$.)

 In what follows we assume that $l\ge 4$.
 Observe next that $G$ can not contain $P_i$, where $i\ge 6$.
 Otherwise replace $P_i$ with $P_2\cup C_{i-2}$ and use
 (\ref{p-p2qi-2}).

 Also $G$ can not contain a cycle $C_i, i\ge 3$ and a path
 $P_j$ for $j\ge 3$.  Indeed, in view of Lemma
 \ref{pathcycleorder} we have the inequality $\Phi_{P_j\cup
 C_i}\prec \Phi_{P_2\cup C_{i+j-2}}$.

 Since $l\ge 4$ it follows that
 $G$ has at least one cycle and
 all paths in $G$ are of length $2$.
 Theorem \ref{uplowmatr=2} implies
 that $G$ contains at most one cycle $C_i\ne C_4$, where
 $i\in [5,6,7]$.  It now follows that $G=H$, where $H$
 satisfies one of the conditions \emph{5-8}.

 \qed

 We now a give the version of Theorem \ref{degreeseq} for the
 subset $\Omega_{\textrm{bi}}(n,k)\subset \Omega(n,k)$ of bipartite graphs.

 \begin{theo}\label{degreeseqbip}
 Let $G$ be a simple bipartite graph of order $n$ with degree sequence
 $d_1=\dots =d_{2k} = 1$ and $d_{2k+1}=\dots =d_n = 2$, where $2\le 2k \le n$,
 i.e. $G\in \Omega_{\textrm{bi}}(n,k)$.  Set $n-2k =
 l$, and assume that $l\ge 2$.  Then (\ref{minmaxmat12}) holds,
 where the graphs $F$ and $H$ depend on $n$ and $k$ as follows.

 \begin{enumerate}
  \item\label{bpl-k<0} When $l-k \le 0$ then $F = lP_3 \cup
      (k-l)P_2$.
 \item\label{bpl-kge0} When $l-k > 0$
 \begin{enumerate}
 \item\label{bpl-k=1,2} If $l-k=1,2$ then
     $F=(k-1)P_3\cup P_{l-k+3}$.
 \item\label{bpl-k=even} If $4\le l-k$ even then either
     $F=F_1 = k P_3\cup  C_{l-k}$ or if $l-k=4$ then
     $F=F_2=(k-1)P_3\cup P_7$.
 \item\label{bpl-k=odd} If $3\le l-k $ is odd, then
     $F=(k -1) P_3\cup P_{l-k+3}$.
 \end{enumerate}
 \item\label{bpl=2}  If $l=2$ then $H=(k-1)P_2\cup P_{4}$.
 \item\label{bpl=3}  If $l=3$ then $H=(k-1)P_2\cup P_{5}$.
 \item\label{bpl=0mod4} If $l\ge 4$ and $l \equiv 0$ (mod
     $4$), then $H = kP_2 \cup \frac{1}{4}lC_4$.
 \item\label{bpl=1mod4} If $l\ge 5$ and $l \equiv 1$ (mod
     $4$), then $H=H_1 = (k-1)P_2 \cup \frac{1}{4}(l-1)C_4
     \cup P_3$ or $H=H_2 = (k-1)P_2 \cup
     \frac{1}{4}(l-5)C_4 \cup P_7$ .
 \item\label{bpl=2mod4} If $l\ge 6$ and $l \equiv 2$ (mod
     $4$), then $H = kP_2 \cup \frac{1}{4}(l-6)C_4 \cup
     C_6$.
 \item\label{bpl=3mod4} If $l\ge 7$ and $l \equiv 3$ (mod
     $4$), then $H=H_1 = (k-1)P_2 \cup \frac{1}{4}(l-3)C_4
     \cup P_5$.
 \end{enumerate}

 Furthermore, if $G\ne F$ then $\Phi_F \prec \Phi_G$ and if
 $G\ne H$ then $\Phi_G \prec \Phi_H$.

 \end{theo}

 \proof The proof of this theorem is similar to the proof of
 Theorem \ref{degreeseq}, and we briefly point out the
 different arguments one should make.
 First, recall that $G\in \Omega(n,k)$ is bipartite, if and
 only if $G$ contains only even cycles.

 We first assume that $G$ is minimal.  Lemma \ref{pathorder} implies that $G$ can not
 contain two paths, such that either each at least length $4$, or one of length $2$ and one of length
 at least $4$.
 Use (\ref{pi-p3qi-3}) to deduce that $G$ can not contain
 $P_i$ for $i\ge 9$.  Also note that $\Phi_{P_7}=\Phi_{P_3\cup
 C_4}$.  By Theorem \ref{uplowmatr=2} $G$ can contain at most
 one even cycle.  Furthermore (\ref{pandpqeven}) yields
 that $G$ can not contain an even cycle and an even path.
 This show that the minimal $G$ must be equal to $F$.

 Assume now that $G$ is maximal.
 Note that in view of Theorem \ref{degreeseq} we need only to
 consider the cases \emph{6} and \emph{8}, i.e. $l\ge 5,\; l\equiv
 1\; mod \;4$ and $l\ge 7,\; l\equiv
 3\; mod \;4$.

 In view of Theorem \ref{uplowmatr=2} can have at most one
 cycle of length $6$, while all the other are of length $4$.
 Lemma \ref{pathorder} implies that one out of any two paths in
 $G$ is $P_2$.  (\ref{p-p2qi-2}) implies that $G$ does not
 contain an even path of length greater than $5$.
 Lemma \ref{pathcycleorder} implies that if $G$ contains an
 even path and a cycle then the length of the even path is $2$.
 (\ref{pi-q4p}) yields that $G$ does not contain an odd path
 of length greater than $8$.  Also one has the equality
 $\Phi_{P_7}=\Phi_{P_3\cup C_4}$ (Lemma \ref{pi-q3pi-3lem}).

 Thus, if an odd path appears in $G$ then we may assume it is
 one of the following: $P_3$, $P_5$ or $P_7$.  First we compare $p_3q_6$ with
 $p_5q_4$.  (\ref{2pcycle01m4}) yields $p_3q_6 \prec q_4 p_5$.
 This establishes the case \emph{8}.
 Next we compare $p_7q_4$ with $p_5q_6$.  Use (\ref{pqminqp})
 to obtain $p_4q_7-q_6p_5=x^5$.  Next use (\ref{comppqminqp}) to show that
 $p_4q_7-q_4p_7=-x^4p_2$.  Hence $q_4p_7-q_6p_5=x^4 p_2 +x^5$.
 Hence $\Phi_{P_7\cup C_4}\succ \Phi_{P_5\cup C_6}$.  This
 establishes \emph{6}.
 \qed

 \section{Expected values of the number of $m$-matchings}

 \subsection{First measure}

%December 26, 2007
% For $n\in\N$ we denote $\an{n}:=\{1,\ldots,n\}$.
 For a set $\cA\subset \R$ denote by $\cA^{p\times q}$  the set
 $p\times q$ matrices $A=[a_{ij}]_{i,j=1}^{p,q}$, where each
 entry $a_{ij}$ is in $\cA$.  For $A=[a_{ij}]\in \R^{n\times n}$
 denote by $\per A$ the \emph{permanent} of $A$, i.e. $\per
 A=\sum_{\sigma\in\rS_n} \prod_{i=1}^n a_{i\sigma (i)}$, where
 $\rS_n$ is the permutation group on $\an{n}$.  Let $A\in
 \R^{p\times q}$ and $m\in \an{min(p,q)}$.  Denote by $\per_m
 A$ the sum of permanents of all $m\times m$ submatrices of
 $A$.

 Denote by $\cG(p,q)$ and $\cG_{\mult}(p,q)$ the set of simple
 bipartite graphs and multibipartite graphs on $p$ and $q$
 vertices in each class, respectively.  W.L.O.G.,
 we can assume that $1\le p\le q$.  We identify the two classes $p$
 and $q$ vertices with $\an{p}$ and $\an{q}$.
 (Sometimes we identify the second class with $q$ vertices with
 $\an{q} +p:=\{p+1,\ldots, p+q\}$.)
 For $G\in \cG(p,q)$ let
 $A(G)=[a_{ij}]_{i,j=1}^{p,q} \in \{0,1\}^{p\times q}$ be the
 $(0,1)$ matrix representing $G$.  Vice versa, any $A\in
 \{0,1\}^{p\times q}$ represents a unique graph $G\in \cG(p,q)$.
 Let $G_1,\ldots,G_r\in
 \cG(p,q)$.  Let $G$ a multi-bipartite graph on the vertices
 $\an{p}\cup \an{q}$,
 whose set of edges is union the set of edges in $G$.  I.e., if
 $e\in \an{p}\times \an{q}$, appears $l$ times in $G$, if and only
 exactly $l$ graphs from $G_1,\ldots,G_r$ contain the edge $e$.
 We denote $G$ by $\vee_{i=1}^r G_i$.
 So $A(G)=[a_{ij}]=\sum_{i=1}^r A(G_i)\in \an{r}^{p\times q}$.
 Vice versa, any $A\in \an{r}^{p\times q}$ corresponds to a
 bipartite multigraph $G$ on the vertices $\an{p},\an{q}$, such that
 $G=\vee_{i=1}^r G_i$, where $G_i\in \cG(p,q)$.  (Usually there
 would be many such decompositions of $G$.)

 In what follows we need the following lemma.

 \begin{lemma}\label{premiden}  Let $p,q,r\in\N$ and assume that
 $G_1,\ldots,G_r\in \cG(p,q)$.  Let $A_i:=A(G_i)\in \{0,1\}^{p\times q}$,
 and denote $A:=\sum_{i=1}^r A_i$.  Let $m\in \an{\min(p,q)}$.
 Then $\per_m A$ is the number of $m$-matchings of
 $G:=\vee_{i=1}^r G_i$, which is equal to the number of
 $m$-matchings obtained in the following way.  Consider $m_1,\ldots,m_r\in
 \Z_+$ such that $m_1+\ldots+m_r=m$.  In each $G_i$ choose an
 $m_i$-matching $M_i$ such that $\cup_{i=1}^r M_i$ is an $m$-matching,
 i.e., $M_i\cap M_j=\emptyset$ for each $i\ne j$.
 \end{lemma}

 \proof Notice that $A$ is the incidence matrix for the multigraph $G:=\vee_{i=1}^r G_i$.
 The permanent of the incidence matrix of a multigraph can be viewed as the number
 of m-matchings of the same graph with multiple edges merged and each edge chosen as many times
 as its multiplicity but not in the same m-matching.
 \qed

 Let $\cS_n$ be the set of all $n\times n$ permutation matrices and
 set
 \[\cS_n^r=\cS_n\times...\times
 \cS_n:=\{(P_1,...,P_r):\;P_1,...,P_r\in \cS_n\}.\]
 Denote by $\cG(2n,r) \subset\cG_{\mult}(2n,r)$ the set of
 simple and multibipartite graphs on $\an{n},\an{n}$ vertices,
 where each vertex has degree $r$.
 Denote by $\Delta(n,r)\subset \{0,1,\ldots,r\}^{n\times n}$ the set of matrices
 with nonnegative integer entries such that the sum of each row
 and column of $A$ is equal to $r$.  That is each $A\in
 \Delta(n,r)$ is the incidence matrix of $G\in
 \cG_{\mult}(2n,r)$.  $G$ is simple if and only if $A\in
 \{0,1\}^{n\times n}$.
 Birkhoff-K\"onig theorem implies that each $A\in \Delta(n,r)$
 is a sum of $r$-permutation matrices.
 \begin{equation}\label{pres1}
 A=P_1+...+P_r,\quad P_1,...,P_r\in \cS_n,
 \end{equation}
 Let $\phi:\cS_n^r\to \Delta(n,r)$ is given by (\ref{pres1}).
 Then for $A\in \Delta(n,r)$ $\phi^{-1}(A)$ is the
 set of all $r$ tuples $(P_1,...,P_r)$ which present $A$.
 Let $\#\phi^{-1}(A)$ be the cardinality of the set
 $\phi^{-1}(A)$.

 View $\cS_n^r$ as a discrete probability space
 where each point $(P_1,...,P_r)$ has the equal probability
 $(n!)^{-r}$.  Then $\phi:\cS_n^r\to \Delta(n,r)$ induces
 the following probability measure on $\Delta(n,r)$:
 \begin{equation}\label{distg}
 P(X_{n,r}=A\in \Delta(n,r))=\frac{\#\phi^{-1}(A)}{(n!)^r}.
 \end{equation}
 Here $X_{n,r}$ is a random variable on the set $\Delta(n,r)$.
 \begin{lemma}\label{mnaver}  Let $1\le r\in N, 1\le m\le n\in N$.
 Assume that the random variable $X_{n,r}\in\Delta(n,r)$ has the
 distribution given by (\ref{distg}).  Then
 \begin{eqnarray}
 &&E_1(m,n,r):= E(\per_m X_{n,r})= \nonumber\\
 &&\frac{1}{(n!)^r}{n \choose m}^2 m!\sum_{m_1,...,m_r\in \Z_+,
 m_1+...m_r=m}\frac{m!(n-m_1)!...(n-m_r)!}{m_1!...m_r!}. \label{mnaver1}
 \end{eqnarray}
 \end{lemma}
 \proof  We first  observe the following equality:

 $$\sum_{P_1,\ldots,P_r\in \cS_n}
 P_1+\ldots+P_r=\sum_{A\in\Delta(n,r)} (\#\phi^{-1}(A))A.$$

 (Just group $P_1+\ldots+P_r$ to $A\in \Delta(n,r)$.)
 Hence
 \begin{equation}\label{expecvalid}
 E(\per_m X_{n,r})=\frac{1}{(n!)^r} \sum_{P_1,\ldots,P_r\in \cS_n}
 \per_m(P_1+\ldots+P_r).
 \end{equation}

 We now compute the right-hand side of (\ref{expecvalid}).
 Each $A=P_1+\ldots+P_r$ we interpret as a regular r-multigraph $ G:=\vee_{i=1}^r
 G_i$.  So $\per_m A$ is the number of total $m$-matchings of $G$.
 It is given by Lemma \ref{premiden}.  We now consider in the
 right-hand side of (\ref{expecvalid}) all terms which contribute
 to a matching $(1,n+1),\ldots, (m,n+m)$.
 (Here $V_1=\{1,...,n\},V_2=\{n+1,...,2n\}$).

 To achieve that we choose
 an $r$ partition $U_1,...,U_r$ of the set $\{1,...,m\}$,
 so that $U_i$ has $m_i\ge 0$ elements.  So $m_1+...+m_r=m$.
 The choice of all such $U_1,...,U_r$ is $\frac{m!}{m_1!...m_r!}$.
 Now once we choose $U_i$, it means that we assumed that
 we choose the edges $(j,n+j), j\in U_i$ from the graph $G_i$ for
 $i=1,\ldots,r$.  This is possible if and only if
 $P_i$
 fixes the elements of $U_i$.  Then there are exactly $(n-m_i)!$
 permutations $P_i$ each of which fixes $U_i$.
 This gives the summand inside the summation in the right-hand side
 of (\ref{distg}).
 Next observe that after we decided that the $m$-matches are chosen
 from the sets $\{1,...,m\} \times \{n+1,...,n+m\}$ then the total
 possible set of $m$-matches for this choice is $m!$.  This gives the
 $m!$ factor outside the summation in the right-hand side
 of (\ref{distg}).  In general we should
 choose two subsets of size $m$ from $V_1$ and $V_2$.  This gives
 the factor ${n \choose m}^2$.  Finally the factor
 $\frac{1}{(n!)^r}$ is the probability of choosing $r$-tuple
 $(P_1,...,P_r)$.
 \qed
 \begin{lemma}\label{lowb}
 Let $2\le r\le m$ be integers.  Let $\mu_1,...,\mu_r$ be $r$
 unique integers satisfying the conditions
 \begin{equation}\label{defmui}
 \mu_i=\lfloor \frac{m}{r}\rfloor,\; i=1,...,k<r,\; \mu_i=
 \lceil \frac{m}{r}\rceil,\; i=k+1,\ldots,r, \; \sum_{i=1}^r\mu_i=m.
 \end{equation}

 %\label{defmk}
 Then
 \begin{eqnarray}
 {m+r-1\choose r-1}\frac{1}{(n!)^{r-2}((n-m)!)^2}
 \prod_{i=1}^r\frac{(n-\mu_i)!}{\mu_i!}\nonumber\ge\\
 E_1(m,n,r)\ge
 \frac{1}{(n!)^{r-2}((n-m)!)^2}
 \prod_{i=1}^r\frac{(n-\mu_i)!}{\mu_i!}\label{lowb1}.
 \end{eqnarray}
 \end{lemma}

 \proof If $r$ divides $m$ then
 $\mu_1=\ldots=\mu_r=\frac{m}{r}$ and (\ref{defmui}) trivially holds
 for any integer $k\in [1,r-1]$.  Assume that $r$ does not divide.
 Then
 \begin{equation}\label{uniqkmu}
 k=r\lceil \frac{m}{r}\rceil -m.
 \end{equation}

 Since the right-hand side of the inequality (\ref{lowb1}) is one of
 the nonnegative summands appearing in the definition
 (\ref{mnaver1}) of $E_1(m,n,r)$ we immediately deduce the lower bound
 in (\ref{lowb1}).

 We next claim the inequality
 \begin{equation}\label{inmimui}
 \frac{(n-m_1)!...(n-m_r)!}{m_1!...m_r!}\le
 \frac{(n-\mu_1)!...(n-\mu_r)!}{\mu_1!...\mu_r!}
 \end{equation}
 for any $r$ nonnegative integer such that $m_1+\ldots+m_r=m$.
 To show this inequality we start with the case $r=2$.  Suppose that
 $0\le a< b-1$ and $a+b=m\le n$.  A straightforward calculation shows:
 $$\frac{(n-a)!(n-b)!}{a!b!}\le\frac{(n-(a+1))!(n-(b-1))!}{(a+1)!(b-1)!}.$$
 (Equality holds if and only if $a+b=n$.)
 Hence the maximum of the left-hand side of (\ref{inmimui})
 on all possible nonnegative integers $m_1,\ldots,m_r$ whose sum
 is $m$ is achieved for $(m_1,\ldots,m_r)$ such that $|m_i-m_j|\le
 1$ for all $i\ne j$.  This implies that the maximum of the left-hand side of (\ref{inmimui})
 is achieved for any permutation of $\mu_1,\ldots,\mu_r$, which
 implies (\ref{inmimui}).
 It is well known that the number of nonnegative integers
 $m_1,\ldots,m_r$ which sum to $m$ is ${m+r-1\choose r-1}$.
   Hence the equality (\ref{mnaver1}) combined with
 (\ref{inmimui}) yields the upper bound in (\ref{lowb1}).

 \qed

 \begin{theo}\label{main}
 Let $2\le r\in\N$. Assume that $1\le m_k \le n_k,
 k=1,...,$ are two strictly increasing sequences of integers such
 that the sequence $\frac{m_k}{n_k}, k=1,...$ converges to $p\in
 [0,1]$. Then
 \[
 \lim_{k\to\infty} \frac{\log E_1(m_k,n_k,r)}{2n_k}=\frac{1}{2}(p \log
 r -p\log p - 2(1-p)\log (1-p)
 +(r-p)\log (1 -\frac{p}{r})).
 \]

 \end{theo}

 \proof Recall Stirling's formula \cite[p. 52]{Fel}:
 \begin{equation}
 n!=\sqrt{2\pi n} \;n^n e^{-n} e^{\frac{\theta_n}{12n}} \textrm{ for some } \theta_n\in (0,1)
 \textrm{ and any positive integer } n.
 \end{equation}
 We will use the following version of Stirling's formula
 $$\sqrt{2\pi n} \;n^n e^{-n} <n!<2\sqrt{2\pi n} \;n^n
 e^{-n}.$$

 Let $\mu_1,\ldots,\mu_r$ be defined by (\ref{defmui}).  We now
 estimate from above and below the terms appearing in
 (\ref{lowb1}) using Stirling's formula.
 \begin{eqnarray*}
 \frac{m-r}{r}<\mu_i<\frac{m+r}{r} \textrm{ for } i=1,\ldots,r,\\
 (\frac{2\pi(m-r)}{r})^{\frac{r}{2}} (\frac{m-r}{r})^{m-r}e^{-m}<
 \prod_{i=1}^r\mu_i!< 2^r(\frac{2\pi(m+r)}{r})^{\frac{r}{2}} (\frac{m+r}{r})^{m+r}e^{-m},\\
 (\frac{2\pi(rn-m-r)}{r})^{\frac{r}{2}} (\frac{rn-m-r}{r})^{rn-m-r}e^{-(rn-m)}<\\
 \prod_{i=1}^r(n-\mu_i)!<
 2^r(\frac{2\pi(rn-m+r)}{r})^{\frac{r}{2}} (\frac{rn-m+r}{r})^{rn-m+r}e^{-(rn-m)},\\
 (2\pi n)^{\frac{r-2}{2}} (2\pi (n-m))
 n^{(r-2)n}(n-m)^{2(n-m)}e^{-((r-2)n+2(n-m))}<\\
 (n!)^{r-2}((n-m)!)^2 <2^r(2\pi n)^{\frac{r-2}{2}} (2\pi (n-m))
 n^{(r-2)n}(n-m)^{2(n-m)}e^{-((r-2)n+2(n-m))},\\
 1\le {m+r-1\choose r-1}< (m+r-1)^{r-1}.
 \end{eqnarray*}

 We now these inequalities in (\ref{lowb1})) to estimate the ratio
 $\frac{1}{2n_k} \log E_1(m_k,n_k,r)$ where
 $$\lim_{k\to\infty}
 m_k=\lim_{k\to\infty} n_k=\infty, \; \lim_{k\to\infty}
 \frac{m_k}{n_k}=p\in [0,1].$$
 First note that for any polynomial $p(x)$ and any $a \in \R$
 $\lim_{k\to\infty} \frac{\log p(m_k+a)}{n_k}=0$.
 Next observe that $\log(x+a)=\log x + O(\frac{1}{x})$ for a
 fixed $a$ and $x\gg 1$.  Let $\frac{m_k}{n_k}=p_k$.
 Our assumptions yield that $\lim _{k\to\infty} p_k=p_k$.
 Then
 \begin{eqnarray*}
 \frac{\log(n_k-\frac{m_k\pm r}{r})^{rn_k-m_k\pm r}e^{-(rn_k-m_k)}}{n_k}
 =(r-p_k +O(\frac{1}{n_k}))(\log n_k +\log (1-\frac{p_k}{r})+\\
 O(\frac{1}{n_k}))-(r-p_k)=
 (r-p_k)(\log n_k +\log (1-\frac{p_k}{r}))-(r-p_k) +o(1),
 \\
 \frac{\log(\frac{m_k\pm r}{r})^{m_k\pm r}e^{-m_k}}{n_k}
 =(p_k + O(\frac{1}{n_k}))(\log n_k +\log p_k-\log r+O(\frac{1}{n_k}))-p_k=\\
 p_k(\log n_k+\log p_k-\log r)-p_k +o(1),\\
 \frac{\log n_k^{(r-2)n_k}(n_k-m_k)^{2(n_k-m_k)}e^{-((r-2)n_k+2(n_k-m_k))}}{n_k}=\\
 (r-2)\log n_k +2(1-p_k)(\log n_k +\log(1-p_k))-r+2p_k.
 \end{eqnarray*}

 Subtract the second and the third term from the first one.
 Note first that the coefficient of $\log n_k$ is
 $(r-p_k)-p_k-(r-2)-2(1-p_k)=0$.  Hence
 \begin{eqnarray*}
 \frac{\log E_1(m_k,n_k,r)}{n_k}=(r-p_k)\log
 (1-\frac{p_k}{r})-(r-p_k)+\\-p_k\log p_k +
 p_k\log r +p_k-2(1-p_k)\log(1-p_k)+r -2p_k +o(1)=\\
 (r-p_k)\log(1-\frac{p_k}{r})-p_k\log p_k +p_k\log r
 -2(1-p_k)\log(1-p_k)+o(1).
 \end{eqnarray*}

 Finally use the continuity of $\log
 x$ to deduce (\ref{limval}). (Here $0\log 0=0$.) \qed

 \subsection{Second measure}

 We now deduce (\ref{limval}) for a standard probabilistic model
 on $\cG_{\mult}(2n,r)$ as given in \cite{LP}.
 Let $\mu\in S_{nr}$ be a permutation on $nr$ elements.
 Let $e_1,...,e_{nr}$ be $nr$ edges going from vertices
 $\{1,...,n\}$ in the group $A$ to vertices $\{1,...,n\}$
 to the group $B$.  We then assume that $e_i$ connects the vertex
 $\lceil\frac{i}{r}\rceil$ in group $A$ to $\lceil\frac{\mu(i)}{r}\rceil$
 in group $B$ for $i=1,...,rn$.  Note that the vertex $i$ in group
 $A$ has $r$ edges labeled $r(i-1)+1,...,ri$.  It is
 straightforward to see that each vertex $j$ in the group $B$
 has $r$ different edges connected to it, i.e. the equation
 $j=\lceil\frac{\mu(i)}{r}\rceil$ has exactly $r$ integers
 $\mu^{-1}(\{j(r-1)+1,...,jr\})$.  Then the probability of such
 graph is given by $\frac{1}{(rn)!}$.  Note if we do not care to
 label the edges, then to a $r$-regular bipartite graph, where
 each two vertices are connected by at most one edge, is
 represented by $(r!)^n$ such permutations $\mu$.
 Indeed any vertex $i$ in the first group has $r$ edges
 labeled $r(i-1)+1,...,ri$ which are connected to it.
 This edges connect to a set of $r$ vertices $T\subset
 \{1,...,n\}$.  Permuting these $r$ edges out of vertex $i$
 between the vertices in the group $T$ has $r!$ choices,
 which are all equivalent.  Repeat this argument for $i=1,...,n$
 to obtain $(r!)^n$ choices which gives rise to the same simple
 graph.  Denote by $\nu(n,r)$ the probability measure on $\cG(2n,r)$
 induced by these method.

 \begin{lemma}\label{mexpcval}  Let $\nu(n,r)$ be the probability
 measure defined above.  Then
 \begin{equation}\label{mexpcval1}
 E_2(m,n,r):=E_{\nu(n,r)}(\phi(m,G))= \frac{{n \choose m}^2
 r^{2m}m!(rn-m)!}{(rn)!}.
 \end{equation}
 \end{lemma}
 \proof  We adopt the arguments of \cite{Sch} to our case.
 First choose subset $\alpha\subset\{1,...,n\}$ of $m$ vertices
 in the group $A$.  There are ${n\choose m}$ choices like that.
 $\alpha$ induces the set $I=\cup_{i\in\alpha} \{r(i-1)+1,...,ir\}$
 of edges of cardinality $rm$.  From $I$ choose a set $J$ of $m$
 edges, so that $e_j,j\in J$ corresponds to the choice of one
 element in the group $\{r(i-1)+1,...,ir\}$, for each $i\in
 \alpha$.  There are $r^m$ of the choices of $J$.  Now we want to
 choose $\mu$ so that $\lceil \frac{\mu(j)}{r}\rceil, j\in J$ will be
 a subset of $m$ distinct elements
 $\beta=\cup_{j\in J}\{\beta_{\lceil\frac{\mu(j)}{r}\rceil}\}\subset \{1,...,n\}$.
 There is $m!{n \choose m}$ such choices of $\beta$.
 Then $\mu(j)\in \{\beta_{\lceil\frac{\mu(j)}{r}\rceil}(r-1)+1,...,
 \beta_{\lceil\frac{\mu(j)}{r}\rceil}r\}$ for each $j\in J$.
 Again there are $r^m$ such choices.  Thus we chose $\mu$ by determining the image
 of the elements in $J$ in $\{1,...,nr\}$, which is denoted by $\mu(J)$.  The rest of the
 of elements $\{1,...,rn\}\backslash J$ is mapped to
 $\{1,...,rn\}\backslash \mu(J)$.   The number of choices here is
 $(nr-m)!$.  Multiply all these choices to get the numerator of
 the right-hand side of (\ref{mexpcval1}).  Divide these number of
 choices by the number of permutations of $\{1,...,rn\}$ to deduce
 the lemma.  \qed

 Using the methods in the proof of Theorem \ref{main} we get the

 \begin{corol}\label{mexpcvalc}
 \begin{eqnarray*}
 \lim_{k\to\infty} \frac{\log E_2(m_k,n_k,r)}{2n_k}=\frac{1}{2}(p \log
 r -p\log p - 2(1-p)\log (1-p)
 +(r-p)\log (1 -\frac{p}{r})), \\
 {\rm if\;} \lim_{k\to\infty} n_k=\lim_{k\to\infty}
 m_k=\infty,\quad {\rm and}\quad
 \lim_{k\to\infty}\frac{m_k}{n_k}=p \in [0,1].
 \end{eqnarray*}
 \end{corol}

 \section{Asymptotic Lower Matching Conjecture}
 For integers $2\le r, 1\le m\le n$ let $\mu(m,n,r)$ be defined
 by (\ref{defmuM}).
 Fix $p\in (0,1]$ and consider two increasing sequences
 $\{m_k\},\{n_k\}$ as in Theorem \ref{main}.
 Let  $\low_r(p)$ be the largest real number (possibly $\infty$) for which one
 always has the inequality
 \begin{equation}\label{defbeta}
 \liminf_{k\to\infty} \frac{\log\mu(m_k,n_k,r)}{n_k}\ge
 \low_r(p), \quad p \in (0,1].
 \end{equation}
 So  $\low_r(p)$ is the limit infimum over all possible values given by the left-hand
 side of (\ref{defbeta}).  Hence $gh_r(p)\ge low_r(p)$ for all
 $p\in [0,1]$.

 The equality (\ref{seqdefp}) and (\ref{nmatchin}) imply the
 equality
 \begin{equation}\label{beq1}
 \low_r(1)=\log \frac{(r-1)^{r-1}}{r^{r-2}}.
 \end{equation}
 (See for details \cite[\S5]{FG} and \cite[\S3]{FKLM}.)
 Hence, in the first version of this paper in 2005 we
 conjectured the \emph{Asymptotic Lower Matching Conjecture},
 abbreviated here as ALMC.

 \begin{con}\label{ascon}  (ALMC) For any $2\le r\in \N, p\in (0,1)$
  $\low_r(p)$ is equal to the right-hand side of (\ref{limval}):
 \[\low_r(p)=p \log
 r -p\log p - 2(1-p)\log (1-p)
 +(r-p)\log (1 -\frac{p}{r})
 \]
 \end{con}

 \begin{theo}\label{valalmcr}  $\low_2(p)=gh_2(p)$ for all
 $p\in[0,1]$, where $gh_2(p)$ is defined by (\ref{defghrp}).
 \end{theo}
 \proof Theorem \ref{uplowmatr=2} yields that
 $$\mu(m,n,2)=\phi(m,C_{2n})={2n-m \choose m}+ {2n-m-1
 \choose m-1}.$$
 Use Stirling's formula as in the proof of Theorem \ref{main}
 to deduce the equality $\low_2(p)=gh_2(p)$.
 \qed

 Friedland and Gurvits \cite[\S5]{FG} have proved the following theorem

  \begin{theo}\label{palrpc}  Let $r\ge 3, s\ge 1$ be integers.
 Let $B_n, n=1,2,\ldots$ be a
      sequence of $n\times n$ doubly stochastic matrices, where each
      column of each $B_n$ has at most $r$-nonzero entries.
      Let $k_n \in [0,n]$, $n = 1,2,\ldots$
      be a sequence of integers with $lim_{n \to \infty}
      \frac{k_n}{n} = p\in (0,1]$.  Then
      \begin{eqnarray}\label{lasrppbp}
      \liminf_{n \to \infty} \frac{\log \per_{k_n} B_n}{2n} \geq
      \frac{1}{2}\left(-p\log p - 2(1-p)\log (1-p)\right) +\\
       \frac{1}{2}\left(
      (r+s-1)\log(1-\frac{1}{r+s})-(s-1+p)\log(1-\frac{1-p}{s})
       \right).\nonumber
      \end{eqnarray}
 Moreover,
 the Asymptotic Lower Matching Conjecture \ref{ascon}
 holds for $p_s=\frac{r}{r+s}, s=0,1,2,\ldots$.
 \end{theo}

 Small lower bounds for $\low_r(p)-gh_r(p)$ for all values of $p\in [0,1]$ are given in
 \cite[\S3]{FKLM}.  Use Stirling's formula, as in the proof of
 Theorem \ref{main} to deduce.

 \begin{prop}\label{finconiminf}  Assume that the inequality
 (\ref{mmatchincon}) holds for all $m\in [2,n]\cap\N$, $3\le r\in \N$
 and all $n\ge N(r)$.  Then ALMC holds.
 \end{prop}

 \section{Maximal matchings in $\cG_{\mult}(2n,r)$ and $\cG(2n,r)$}

 \begin{prop}\label{maxnummatmulr}  Let $G=(V_1\cup V_2,E)$ be a bipartite multigraph
 where $V_1, V_2$ are the two groups of the set of vertices.
 Let $\#V_1=n$ and assume that the degree of each vertex in $V_1$ is $r\ge 2$.
 Then
 \begin{equation}\label{upmatchmubgrph}
 \phi(m,G)\le {n \choose m} r^m \textrm{ for each } m=1,\ldots,n.
 \end{equation}
 Assume that $\#V_2=n$.  Then for $m\ge 2$ equality holds if and only if $G=nH_r$, i.e.
 $A(G)=r I_n$.  In particular (\ref{Mmnrfor}) holds.
 \end{prop}

 \proof  Let $M\subset E$ be an $m$-matching.
 Then $M$ covers exactly $U\subset V_1$ vertices of cardinality
 $m$.  Then number of choices of $U$ is $n \choose m$.
 Let $v\in U$.  Then $v$ can be covered by $r$ edges.  Hence
 (\ref{upmatchmubgrph}) holds.

 Suppose that $m\ge 2$ and $\#V_2=n$.
 Let $w\in V_2$ and assume that $w$ connected to two distinct
 vertices $v_1,v_2\in V_1$ by the edges $e_1,e_2$.  Then these
 two edges can not appear together in any $m$-matchings.  Hence
 for this $G$ one has a strict inequality in
 (\ref{upmatchmubgrph}).  Thus, if $\#V_2=n$ and $m\ge 2$ equality holds in
 (\ref{upmatchmubgrph}) if and only if $G=nH_r$.  \qed

 The inequality (\ref{upmatchmubgrph}) for $G\in \cG(2n,r)$ was
 used in \cite{FKLM}.  In the first version of this paper we
 conjectured that $\Lambda(m,n,r):=\max _{G\in \cG(2n,r)}
 \phi(m,G)$ is achieved for the maximal graph $qK_{r,r}$ , i.e.
 disjoint unions of $q$ complete bipartite graphs on $2r$
 vertices, if $n\equiv 0\;mod \;4$.

  We state a generalization of the conjecture (\ref{upmatcon}) for $\cG(2n,r)$
 when $n$ is not divisible by $r$:
 \begin{equation}\label{conjmaxmatreg}
 \phi(m,G)\le \phi(m,\lfloor \frac{n}{r}\rfloor K_{r,r}\cup
 (n-r\lfloor \frac{n}{r}\rfloor) H_r) \textrm{ for any } G\in \cG(2n,r).
 \end{equation}

 Theorem \ref{uplowmatr=2} yields that the validity of the
 conjecture (\ref{conjmaxmatreg}) for $r=2$.
 See \cite{FKLM} for the asymptotic version of the
 conjectured inequality (\ref{conjmaxmatreg}).

 %------------------------------------------------------------------------------
\section{Computational results}

%-----------------------------------------------------------
\subsection{The Lower Matching Conjecture for finite graphs}

For small $r$-regular bipartite graphs on $2n$ vertices we have
tested the following finite analogue of the lower matching
conjecture.
\begin{equation}\label{slowmcon}
    \phi(G,m)\ge \varphi(n,r,m)=\left ( 1+\frac{1}{r n}  \right)^{r n-1}
    (1-\frac{m}{rn})^{rn-m} (\frac{mr}{n})^m {n\choose m}^2
\end{equation}
Note that as $n$ grows this bound is asymptotically exact for
1-edge matchings, and the convergence is faster for larger $r$.

In order to test the conjecture we computed the matching
generating polynomials for all bipartite regular graphs on
$2n\leq 20$ vertices and compared with the bound. The bound
held for all such graphs.

For $2n \geq 21$ the number of bipartite regular graphs is too
large for a complete test of all graphs, the computing time for
each graph also grows exponentially, so we instead tested  the
conjecture for graphs of higher girth. The combinations of
degree and girth are given in table \ref{dg-tab}. Again the
conjecture held for all such graphs.

\begin{figure}[b]
    \begin{tabular}{|r|c|c|c|c|c|c|c|c|c|c|c|c|c|c|c|c|c}
	\hline
	$2n$   & 22 & 24 & 26 & 28 & 30 & 32 & 34 & 36  \\
	\hline
	$r=3$ & 6  & 6  & 6  & 6  & 6  & 7  & 7  & 7   \\
	$r=4$ &    &    & 6  & 6  & 6  & 6  & 6  &     \\
	
	\hline
    \end{tabular}

    \caption{Lower bound for the girth of the regular bipartite graphs
    of order greater than 20 used in our tests. An empty entry
    means that no graphs of that order and degree were used.}
\end{figure}\label{dg-tab}

%-----------------------------------------------------------
 \subsection{The Upper Matching Conjecture for Cubic graphs}

 We have checked the upper matching conjecture for $r=3$ and $2n$
 up to 24 by computing the matching generating polynomials for
 all connected bipartite cubic graphs, up to an isomorphism, in
 this range. For $2n=6$ and $2n=8$ there is only one cubic
 bipartite graph of the given size: $K_{3,3}$ and the
 3-dimensional hypercube $Q_{3}$ respectively. For $2n=10$ there
 are two graphs to consider and they turn out to have
 incomparable matching generating functions. The first graph
 $G_{1}$ is shown in Figure \ref{fig:g10} and the second graph
 is the 10 vertex M\"{o}bius ladder $M_{10}$. ($M_{10}$ consists
 of two copies of path of length $5$: $1-2-3-4-5$, denoted by
 $(P_5,1)$ and $(P_5,2)$, where first one connects $(i,1)$ and
 $(i,2)$ by an edge for $i=1,\ldots,5$, and then one connects
 $(1,1)$ with $(5,2)$ and $(1,2)$ with $(5,1)$.)

Their matching generating polynomials are:
 \begin{eqnarray*}
 \psi(x,G_{1}):=1+15x+75x^2+145x^3 +96x^4 +12x^5,\\
 \psi(x,M_{10}):=1+15x+75x^2+145x^3 + 95x^4 + 13x^5.
 \end{eqnarray*}

For $2n$ from 12 to 24 the extremal graphs, with the maximal
$\phi(l,G)$, are for the form

\begin{equation}\label{3regmax12-24}
     \begin{array}{ll}
    \frac{2n}{6}K_{3,3} & \textrm{if}\ 6|2n \\
    \frac{2n-8}{6}K_{3,3}\bigcup Q_{3} & \textrm{if}\ 6|(2n-2) \\
    \frac{2n-10}{6}K_{3,3}\bigcup (G_{1}\ \textrm{or}\ M_{10}  ) &
\textrm{if}\ 6|(2n-4) \\
     \end{array}
\end{equation}

So for $2n=10,22$ we do not have a unique extremal graph, which
maximizes all $\phi(l,G)$. It seems natural to conjecture that
the three graph families given here together make up all the
extremal graphs for all $n$.

\begin{figure}[tbp]
     \includegraphics*[width=6cm,height=4cm]{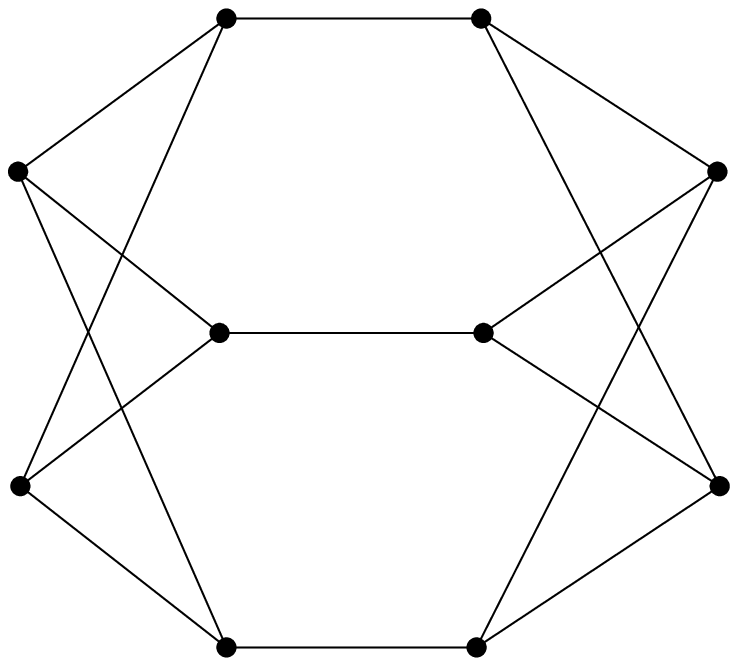}
     \caption{$G_{1}$}
     \label{fig:g10}
\end{figure}

%-----------------------------------------------------------
 \section{Exact values for small matchings}

 \begin{theo}\label{exactsmall}
     Assume that $G$ is a biparite $r$-regular graphs on $2n$ vertices
     and that $G$ contains $a_{4}(G)$ 4-cycles, then
     \begin{enumerate}
 	\item
 	$\phi(G,1) = r n$ 	
 	\item	
 	$\phi(G,2) ={nr\choose 2} - 2n {r \choose 2} =\frac{r n(r n
 -(2r-1) )}{2}$

 	\item	
 	$\phi(G,3) = {nr \choose 3 }-2n{r\choose
 3}-nr(r-1)^{2}-2n{r\choose
 	2}(nr-2r-(r-2)) $ 	
 	\item $\phi(G,4) =p_{1}(n,r)+a_{4}(G)$
 where
 \begin{equation}\label{defp1nr}p_{1}(n,r)=
 	\frac{n^{4}r^{4}}{24}+
 	\frac{n^{3}r^{3}}{4}(1-2r)+
 	\frac{n^{2}r^{2}}{24}\left(19-60r+52r^{2}\right)+
 	nr\left(\frac{5}{4}-5r+7r^{2}-\frac{7r^{3}}{2}\right).
 \end{equation}

     \end{enumerate}

 \end{theo}
 \begin{proof}
     \begin{enumerate}
 	\item This is just the number of edges in $G$. 	
 	\item There are ${nr\choose 2}$ 2-edge subsets of $E(G)$.
 Such a subset is not a matching if it forms a three vertex path
 $P_{3}$.
 	Given a $P_{3}\subset G$ we call the vertex of degree 2 the
 	root. The number of $P_{3}$'s in $G$ is $2n {r \choose 2}$,
 	since there are $2n$ choices for the root vertex and at
 that
 	vertex there are ${r \choose 2}$ ways of choosing  two
 edges. 	
 	\item As in the previous case three edges in $G$ can be
 	chosen in $ {nr \choose 3 }$ ways. There are three
 three-edge
 	subgraphs which are not a matching, depicted in
 Figure \ref{bad3}.
 	The number of 4-vertex stars, $2n{r\choose 3}$, is counted
 as in
 	the previous case. The number of $P_{4}$'s is
 $nr(r-1)^{2}$,
 	since the middle edge can be chosen in $nr$ ways and the
 two
 	remaining edges in $r-1$ ways each.	
 	The number of subgraphs $P_{3}\cup K_{2}$ is $2n{r\choose
 	2}(nr-2r-(r-2))$, since the $P_{3}$ can be chosen as in the
 	previous case, and the $K_{2}$ can be chosen among the
 $(nr-2r-(r-2))$
 	edges which are not incident with any of the vertices in
 the
 	$P_{3}$. 	
 	\item  Let $\cE_4(G)$ be the subset of all subgraphs of $G\in
 \cG(2n,r)$ consisting of $4$ edges.
 Then $\#\cE_4(G)={nr \choose 4 }$.  For $H\in\cE_4(G)$ let
 $l(H)\ge 0$ be the number $P_3$ subgraphs of $H$.
 $H\in \cE_4(G)$ is a matching if and only $l(G)=0$.
 There are  $2n{r \choose 2 }{nr-2 \choose
 	2 }$ graphs $H \in \cE_4(G)$ which contain at least one $P_{3}$
 	with a specified root vertex, since there are $2n$ ways to
 place the root of a $P_{3}$ and ${nr-2 \choose
 	2 }$ ways to choose the remaining two edges.
 Note that
 $2n{r \choose 2 }{nr-2 \choose 2 }=\sum_{H\in\cE_4(G),l(H)\ge 1} l(H)$.
 Thus, the correct of $4$-matches is
 \begin{equation}\label{numb4mat}
 {nr \choose 4}-2n{r \choose 2 }{nr-2 \choose
 	2 }+\sum_{H\in \cE_4, l(H)>1}(l(H)-1).
 \end{equation}
 	In Figure \ref{bad4} we display all subgraphs $H$ with
 $l(H)>1$.
 	The number of copies of each graph and its number of
 $P_{3}$'s is
 	\begin{enumerate}
 	    \item[S1] Number: $ 2n{r \choose 4}  $ ,$P_{3}$'s
 ${4\choose 2}$
 	    \item[S2] Number: $ 2n{r \choose 3}3(r-1)  $ ,$P_{3}$'s
 $1+{3\choose 2}$
 	    \item[S3] Number: $ 2n{r\choose 3}(nr-4r+3))  $
 ,$P_{3}$'s ${3\choose 2}$
 	    \item[S4] Number: $ 2n{r\choose 2}(r-1)^{2}-4a_{4}(G)
 $ ,$P_{3}$'s 3
 	    \item[S5] Number: $ a_{4}(G) $ ,$P_{3}$'s 4
 	    \item[S6] Number: $ n(n-2){r\choose
 2}^{2}-\frac{1}{2}(\# S2)  $ ,$P_{3}$'s 2
 	    \item[S7] Number: $ 2{n\choose 2}{r\choose
 2}^{2}-2a_{4}(G)-(\#S4)  $ ,$P_{3}$'s 2
 	    \item[S8] Number:
 $(nr(r-1)^{2}-4a_{4}(G))(nr-4r+3)+4a_{4}(G)(nr-4r+4)  $
 ,$P_{3}$'s 2
 	\end{enumerate}		
 Use the above formulas in (\ref{numb4mat}) to obtain
 a rather messy expression for $\phi(G,4)$.
 After some simplification we obtain is the formula we have in
 the theorem.
     \end{enumerate}
 \end{proof}

 \begin{figure}[tbp]
      \includegraphics*[width=5cm,height=3cm]{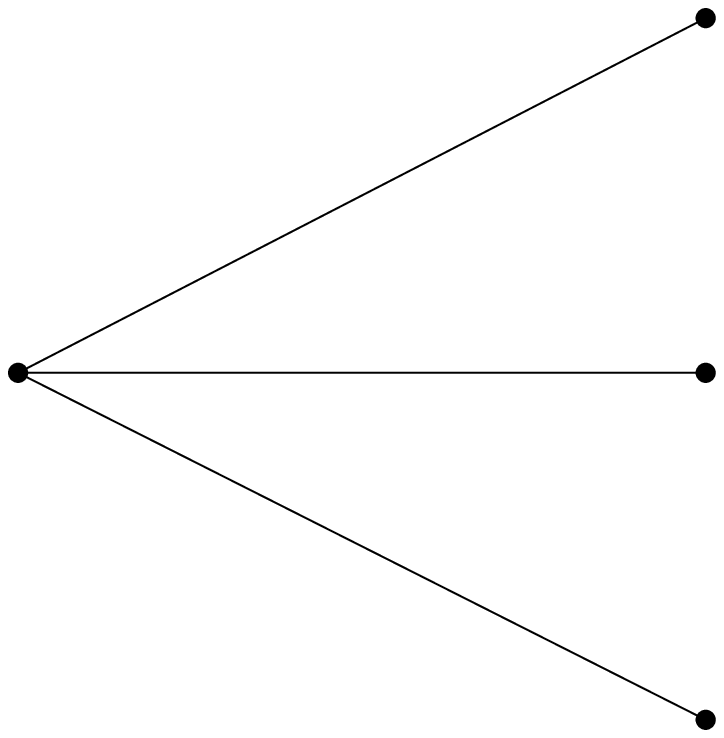}
      \includegraphics*[width=5cm,height=3cm]{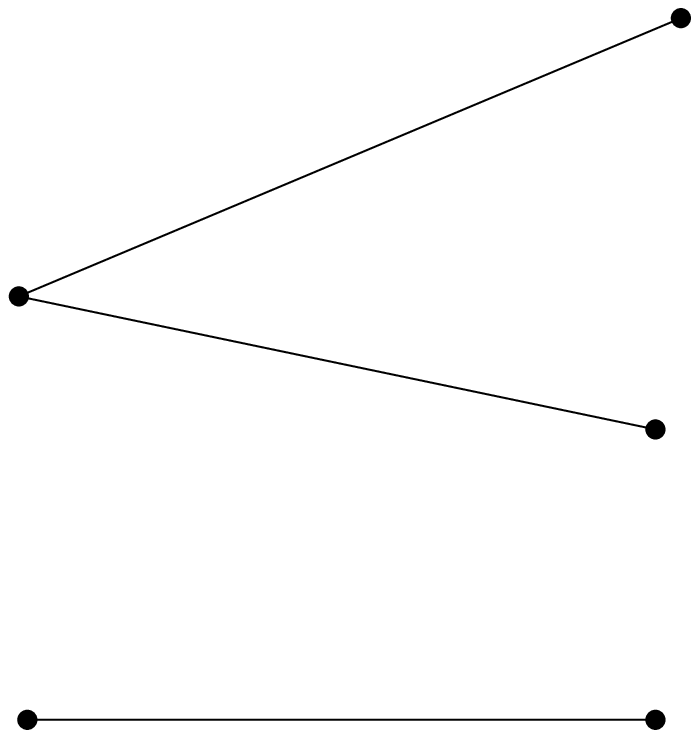}
      \includegraphics*[width=5cm,height=3cm]{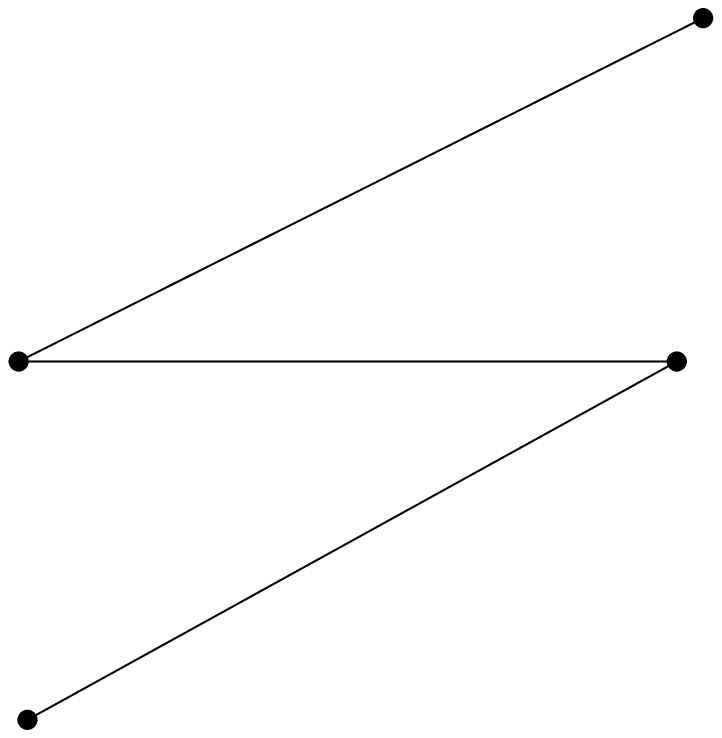}
      \caption{The 3 edge subgraphs}\label{bad3}
 \end{figure}

 \begin{figure}[tbp]
      \includegraphics*[width=14cm,height=14cm]{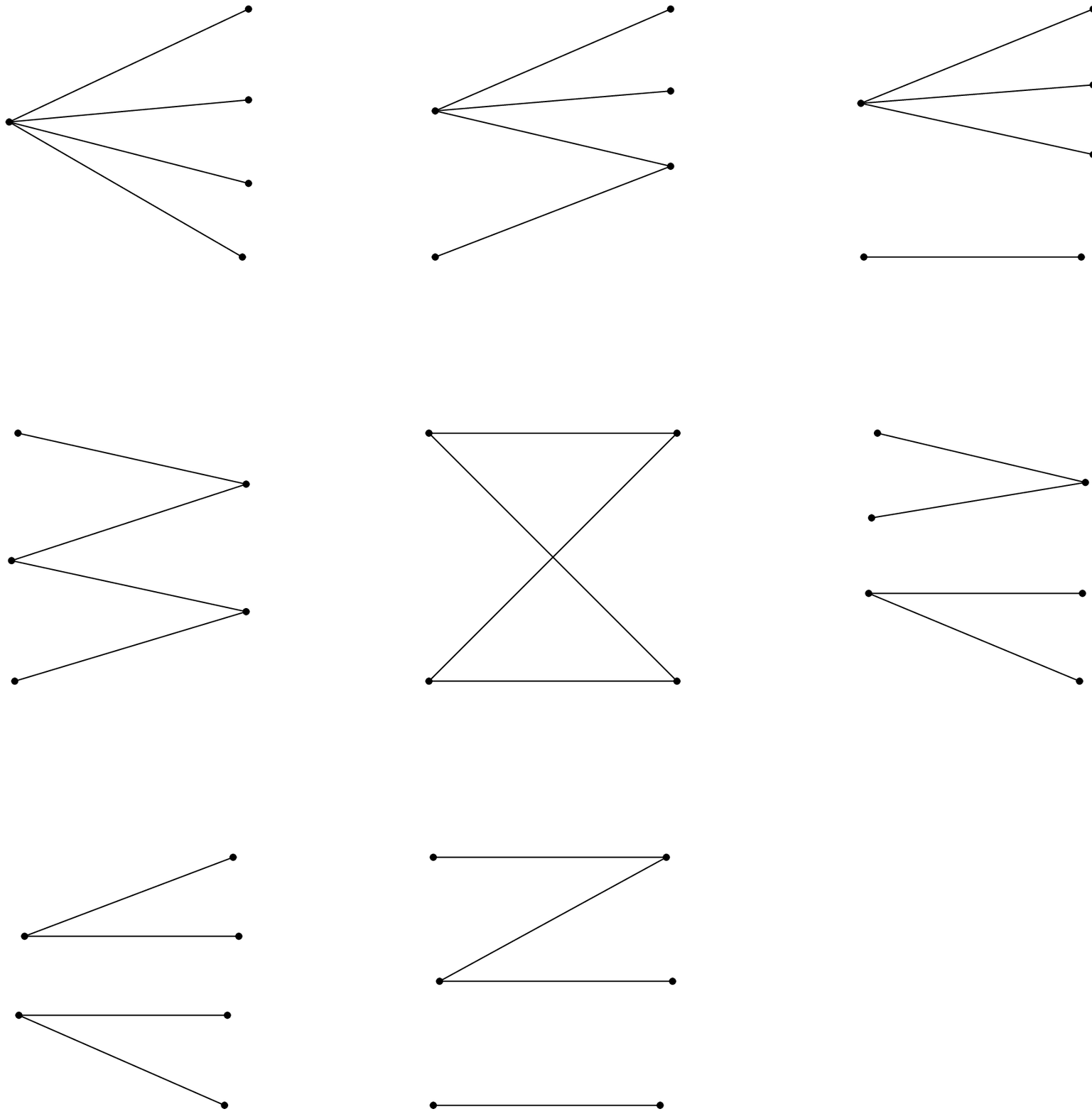}
      \caption{The 4 edge subgraphs}\label{bad4}
 \end{figure}

 If we compute the limits of $\frac{\phi(G,m)}{\varphi(n,r,m)}$
 for the values of $m$ used in Theorem \ref{exactsmall} we find
 that
 $$\lim_{n\rightarrow \infty}\frac{\phi(G,1)}{\varphi(n,r,1)}= 1$$
 $$\lim_{n\rightarrow \infty}\frac{\phi(G,2)}{\varphi(n,r,2)}=
 \frac{e}{2}=1.359$$
 $$\lim_{n\rightarrow
 \infty}\frac{\phi(G,3)}{\varphi(n,r,3)}=\frac{2e^{2}}{9}=1.642\ldots$$
 $$\lim_{n\rightarrow
 \infty}\frac{\phi(G,4)}{\varphi(n,r,4)}=\frac{3e^{3}}{32}=1.883\ldots$$
 This indicates that there exists some stronger form of the
 lower bound for finite graphs, but if the ALMC is true this
 additional factor will be subexponential in $n$, possibly just
 a function of $m$.

 In the expression for $\phi(G,4)$ the number of 4-cycles
 appeared as the first structure in the graph, apart from $n$
 and $r$, which affects the number of matchings. The maximum
 possible value of $a_{4}(G)$ can be found.
 \begin{lemma}\label{upbd4cyc}  Let $G$ be an $r$ regular bipartite graph on
 $2n$ vertices, with $r\ge 2$.  Then
 \begin{equation}\label{upbd4cyc1}
 a_4(G)\le \frac{nr(r-1)^2}{4}.
 \end{equation}
 Equality holds if and only if $n=qr$ and
      $G$ is the disjoint union of $q$ $K_{r,r}$.
 \end{lemma}
 \begin{proof}
     Given an edge $e$ in $G$,
      the largest number of 4-cycles which can
     contain $e$ is $(r-1)^{2}$. Indeed,
     the number of $P_{4}$'s
     which contain $e$ is $(r-1)^2$.  Each $P_4$ can be
     completed to a $4$ cycle
      if an only if
     $e$ is an edge in a connected component of $G$ equal to
     $K_{r,r}$.  Since $G$ has $nr$ edges and each $4$ cycle
     consists of $4$ edges we deduce the inequality
     (\ref{upbd4cyc1}).  Assume equality in (\ref{upbd4cyc1}).
     Then every edge belongs to a
     $K_{r,r}$ component of $G$.  Hence $G=qK_{r,r}$.
     \qed
 \end{proof}

 This has some simple but nice corollaries.
 \begin{corol}
     The upper and lower matching conjectures are true for $m\leq4$.
 \end{corol}
 In \cite{wormald:1978} the distribution of the number of short
 cycles in a bipartite random regular graph was determined, and applying
 that result here we find that,
 \begin{corol}
     For  random graphs from $\mathcal{G}(2n,r)$ we have that
     $\phi(4,G)-p_{1}(n,r)$ converges in distribution  to a Poisson random
     variable with expectation $\frac{(r-1)^{4}}{4}$.
 \end{corol}
 This means that the expected number of 4-edge matchings in a
 random graph is only a fixed constant larger than the minimum
 possible, and also only a fixed constant larger than the lower
 matching conjecture

 \bibliographystyle{plain}

 \end{document}